\documentclass[11pt]{article}
\usepackage{amsfonts, amsmath}
\usepackage{amssymb}
\usepackage{amsthm,amscd}
\usepackage{enumerate}
\usepackage{tikz}
\usepackage{lipsum}
\usepackage{subfig}
\usepackage{lineno}
\usepackage{caption}

\newtheorem{theorem}{Theorem}[section]
\newtheorem{corollary}[theorem]{Corollary}
\newtheorem{proposition}[theorem]{Proposition}
\newtheorem{conjecture}[theorem]{Conjecture}

\numberwithin{equation}{section}
\numberwithin{figure}{section}

\newcommand{\T}{\mathbb T}

\newcommand{\R}{\mathbb R}
\newcommand{\Pa}{\mathcal P}

\newcommand{\bna}{\begin{eqnarray}}
\newcommand{\ena}{\end{eqnarray}}
\newcommand{\ba}{\begin{eqnarray*}}
\newcommand{\ea}{\end{eqnarray*}}
\newcommand{\bs}[1]{}

\def\p{{\bf p}}

\def\pn{{\bf p =(p_1, \dots, p_n) }}
\def\q{{\bf q}}

\def\r{{\bf r}}
\def\s{{\bf s}}

\usepackage[margin=1in]{geometry}
\begin{document}
\title{The Isostatic Conjecture \thanks{This work was partially supported by the National Science Foundation Grant DMS-1564493 for Connelly, Solomonides and Yampolskaya, and National Science Foundation Grant DMS-1564473 for Gortler.}\\
}

\author{Robert Connelly,  Steven J. Gortler, 
Evan Solomonides, and Maria Yampolskaya}
\maketitle 
{\footnotesize \noindent Department of Mathematics (for Connelly, Solomonides, and Yampolskaya)\\
 \footnotesize Cornell University \\
 \footnotesize Ithaca, NY 14853\\
 \footnotesize mailto:  connelly@math.cornell.edu}\\
 
 {\footnotesize \noindent School of Engineering and Applied Sciences (for Gortler)\\
 \footnotesize Harvard University \\
 \footnotesize Cambridge, MA 02138}

\begin{abstract}  We show that a jammed packing of disks with generic radii, in a generic container, is such that the minimal number of contacts occurs and there is only one dimension of equilibrium stresses, which have been observed with numerical Monte Carlo simulations.  We also point out some connections to packings with different radii and results in the theory of circle packings whose graph forms a triangulation of a given topological surface. 


\end{abstract}
\section{Introduction} \label{section:introduction}

Granular material, made of small rocks or grains of sand, is often modeled as a packing of circular disks in the plane or round spheres in space. In order to analyze the internal stresses that resolve external loads, there is a lot of interest in the distribution of the stresses in that material.  See, for example, \cite{stress-distribution-I, stress-distribution-II, stress-distribution-III, missing-stress-equation}.  A \emph{self stress} is an assignment of scalars to the edges of the graph of contacts such that at each vertex (the disk centers) there is a vector equilibrium maintained.  One property that has come up in this context is that, when a packing is jammed in some sort of container, there is necessarily an internal self-stress that appears.  It seems to be taken as a matter of (empirical) fact that when the radii of the circles (or spheres) are chosen generically, there is only one such self stress, up to scaling.  In that case, one says that the structure is \emph{isostatic}.  This statement seems to be borne out in many computer simulations, since it is essentially a geometric property of the jammed configuration of circular disks.  See, for example, the work of J-N Roux \cite{Roux-isostatic}, Atkinson et. al.  \cite{Torquato-isostatic, Torquato-mono}, and \cite{basics}.  When the disks all have the same radius, and thus are non-generic, for example, it quite often turns out that the packing is not isostatic.  Here, we refer to the (mathematical) statement that when the radii and lattice are generic, the packing has a single stress up to scaling, as the \emph{isostatic conjecture}.  Note that when the radii of the packing disks are chosen generically,  this does not imply that the coordinates of the configuration of the centers of the disks are generic.  There is a wide literature on the rigidity of frameworks, when the configuration is generic, for example from the basic results in the plane starting with Laman \cite{Laman}, and more generally as described in Asimow and Roth \cite{Asimow-Roth-I, Asimow-Roth-II}.  But if the graph of the packing has any cycle of even length, the corresponding edge lengths between the vertex centers will not be generic since the sum of the lengths of half of the edges will be the same as the sum of the lengths of the other half of the edges.   The configuration of the centers will not be generic either, since if they were, the edge lengths of a cycle would also be generic. 

There are many different instances when the isostatic conjecture could be posed. For example, one could enclose a collection of disks with fixed radii inside a polygon and squeeze the shape of the polygon until the packing inside is jammed.  It can happen that there is an occasional disk that is not fixed to the others, which we call a \emph{rattler}, and in that case we ignore it, since it does not contribute to the self stress lying in the packing.  See also \cite{Torquato-isostatic, Torquato-mono} for the effect of rattlers on the density.  Another example, and one that we will investigate in our study, is when a packing is periodic with a given lattice determining its overall symmetry. We then increase the radii uniformly, keeping a fixed ratio between every pair of radii, maintaining the genericity.  In dimension three and higher, we do not provide any method  to prove that jammed packings are isostatic. We rely very heavily on two dimensional techniques. 

Another interesting aspect of the ideas here is that we connect some of the principles of the rigidity theory of jammed packings as in the work of Will Dickinson et al. \cite{Connelly-Vivian}, Oleg Musin and Anton Nikitenko \cite{Musin},  \cite{Connelly-Smith}, and \cite{Connelly-Vivian}, with another theory of analytic circle packings as in the book of Kenneth Stephenson \cite{Stephenson}.  There are essentially two seemingly independent methods of creating a circle packing.  One is created by modeling the disks as having fixed ratios and increasing the packing density until they jam, and the other is based on an idea that goes back to at least Koebe, Andreev and Thurston \cite{Thurston}, where the graph of the packing is a predetermined a triangulation.  Here we use an extension where some of the distances between disks are determined by an inversive distance, defined later, between circles by Ren Guo in \cite{Guo}.


\section{Rigidly jammed circle packings} \label{section:jammed}

We need a container for our packings.  For the sake of simplicity and because of the lack of boundary effects, we will use the $2$-dimensional torus $\T^2=\R^2/\Lambda$ regarded as the Euclidean plane  $\R^2$ modulo the fixed integer lattice $\Lambda$ as the container.  A packing ${{\bf P}}$ in $\T^2$ is a finite union of labeled circular disks with disjoint interiors.  We say that ${{\bf P}}$ is \emph{locally maximally dense} if there is an $\epsilon > 0$ such that for any other  packing ${{\bf Q}}$ with corresponding radii in the same ratio and $|{{\bf Q}}-{{\bf P}}| < \epsilon$, then $\rho({{\bf Q}}) \le \rho({{\bf P}})$. The density of ${{\bf P}}$ is $\rho({{\bf P}}) = \sum_i A(D_i)/A(\T^2)$, where $A()$ is the usual area function.  The distance between packings is regarded as the distance between the vectors of the centers and radii of the packing disks.  In other words, near the packing ${{\bf P}}$, except for translates,  we cannot increase the packing radii uniformly and maintain the packing constraints. 

A first process is that we can ``inflate" the packing disks uniformly until some subset of the packing disks jam and prevent any other expansion.  This process is called a ``Monte Carlo" method in Torquato et. al. \cite{Donev-Torquato-Connelly, Torquato-tricusp}.  Let $\r=(r_1, \dots, r_n)$ be the radii of the corresponding packing disks ${{\bf P}} = (D_1, \dots, D_n)$.  The idea is to continuously increase the radii to for $t> 1$, and at the same time continuously deform the packing to ${{\bf P}}(t)$ so that the radius of $D_i$ is $tr_i$, until no further increase in $t$ is possible.  Then the resulting packing ${{\bf P}}(t_1)$ will be locally maximally dense.  We would like to say that ${{{\bf P}}}(t_1)$ is rigid or jammed.  But there is a problem with rattlers as in Figure \ref{fig:rattler}.   Considering the rigidity of the packing, we just discard the rattlers.
 \begin{figure}[!htb]
    \centering
        \includegraphics[width=0.4\textwidth]{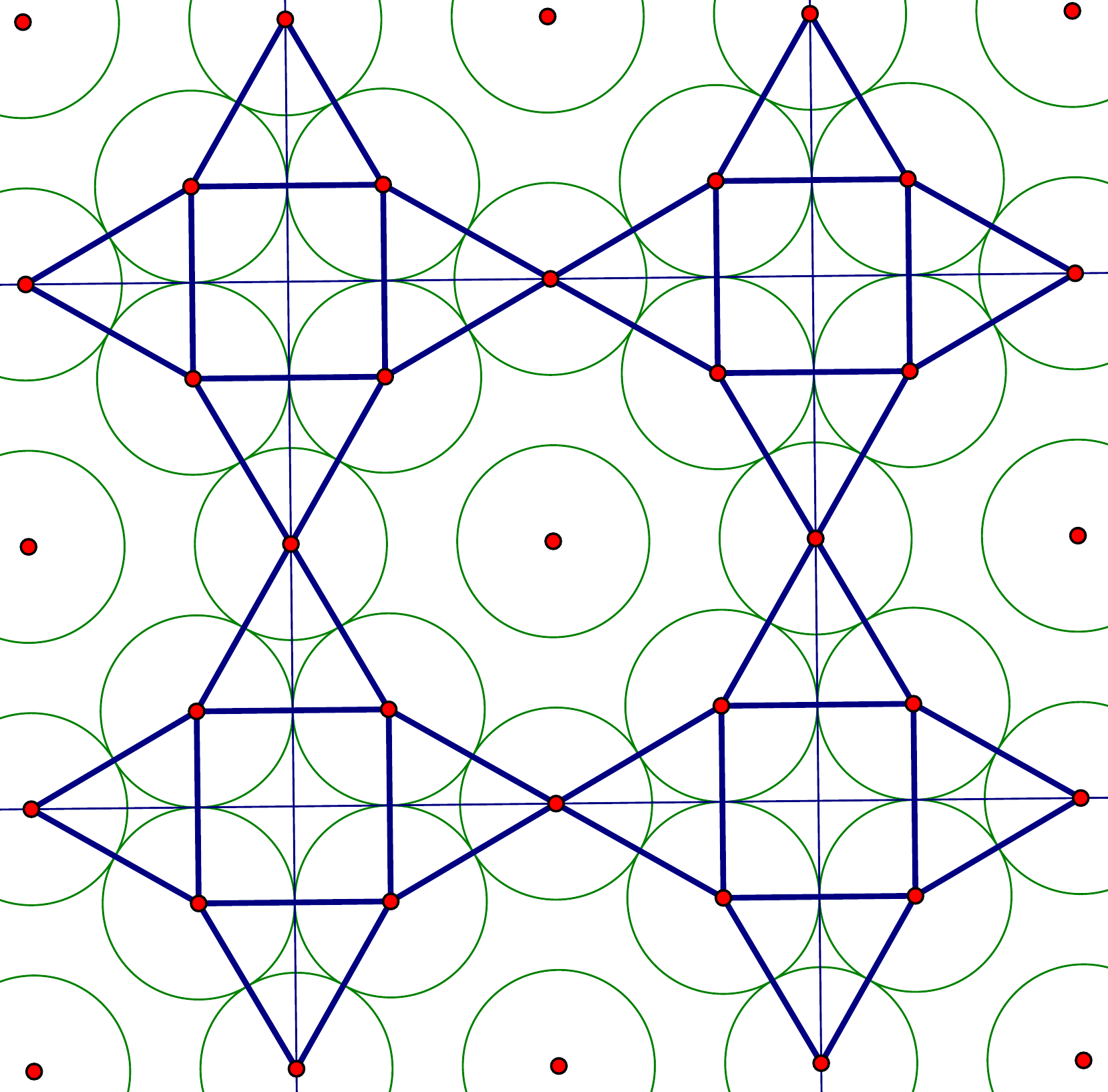}%
        \captionsetup{labelsep=colon,margin=2cm}
         \caption{This is a maximally dense packing of $7$ disks  in a square torus,   with a rattler in the middle, found by Musin and Nikitenko  in \cite{Musin}. The thin horizontal and vertical lines  outline fundamental regions of the torus.}
\label{fig:rattler}
    \end{figure}
     
We need some tools to determine the rigidity of packings.  Given a packing ${{\bf P}}$ in a torus, define a packing graph $G({{\bf P}})$, where vertices are the centers of the packing disks, and edges connect centers whose disks touch, as in Figure \ref{fig:rattler}.  Note that in the torus $G({{\bf P}})$ may have loops and multiple edges, but in all our calculations, we work in the universal cover where $G({{\bf P}})$ has no loops or multiple edges.  We are effectively working with the equivalence classes of lattices $\p_i +\Lambda$ in the plane.  When there is an (oriented) edge joining two vertices of $G({{\bf P}})$ this can always be represented as a well-defined vector $\p_i-\p_j$ in the plane $\R^2$.

Let $\p_i$ be the center of disk $D_i$.  Let $\p'=(\p_1',\dots, \p_n')$ be a corresponding sequence of vectors, where $\p_i' \in \R^2$ is in the tangent space of $\T^2$ at $\p_i$.  We say that $\p'$ is an \emph{infinitesimal flex} of $G({{\bf P}})$, if for each edge $\{i,j\}$ of $G$,
\begin{equation}\label{eqn:inf-flex}
(\p_i-\p_j)\cdot (\p'_i-\p'_j) \ge 0.
\end{equation}
 We say that $\p'$ is a \emph{trivial infinitesimal flex} of $G({{\bf P}})$, if all the  $\p'_i$ are the same vector, for $i=1\dots, n$.  We say that a packing ${{\bf P}}$ is \emph{(locally) rigid or collectively jammed} if the only continuous motion of the packing (preserving the radii) is by translations.  We say that a packing ${{\bf P}}$ is \emph{infinitesimally rigid} if every infinitesimal flex is trivial.
 
\begin{theorem}\label{theorem:Danzer} A packing ${{\bf P}}$ in a torus $\T^2$ is collectively jammed if and only if it is infinitesimally rigid.  Furthermore ${{\bf P}}$ is locally maximally dense if and only if there is a subset of the packing that is collectively jammed.
\end{theorem}

We call the packing disks, minus the rattlers, a \emph{spine} of the jammed packing.

A proof of this statement and Corollary \ref{cor:rigid-count}, later, can be found in \cite{basics, Donev-Torquato-Connelly}. The idea goes back at least to Danzer \cite{Danzer}, and it works for all dimensions for all compact surfaces of constant curvature, except, interestingly, the proof of the ``only if" part fails for surfaces of positive curvature, such as the sphere.  The infinitesimal flex can be used to define a continuous flex for the whole configuration and the higher-order terms work in our favor.  The ``if" part is standard.

Another closely related property is when the lattice $\Lambda$ itself is allowed to move.  Following  \cite{Donev-Torquato-Connelly}, we say that a packing in a torus $\T^2=\R^2/\Lambda$  is \emph{strictly jammed} if it is rigid allowing both the configuration and the lattice $\Lambda$ to move locally with the constraint that the total area of $\T^2$ not increase.  Here we will concentrate mostly on collective jamming.

\section{Basic rigidity of tensegrities} \label{section:rigidity}

The rigidity part of Section \ref{section:jammed} can be rephrased in the language of frameworks and tensegrities.  One is given a finite configuration of points $\pn$, in our case in the $2$-torus $\T^2$, and a \emph{tensegrity graph} $G$, where each edge of $G$ is defined to be a \emph{cable}, which is not allowed to increase in length, or a \emph{strut}, which is not allowed to decrease in length, and a \emph{bar}, which is not allowed to change in length.  In our case the packing graph consists entirely of struts.  

The next tool we need is the concept of a \emph{stress} for the graph,  which is just a scalar $\omega_{ij}=  \omega_{ji}$ assigned to each edge $\{i,j\}$ of $G$.   We say that a stress $\omega = (\dots, \omega_{ij}, \dots)$ is an \emph{equilibrium stress} if for each vertex $i$ of $G$ the following holds
\[
\sum_{j} \omega_{ij}(\p_i-\p_j)=0,
\] 
where  $\omega_{ij}=0$ for non-edges $\{i,j\}$.  Furthermore, we say that a stress for a graph $G$ is a \emph{strict proper stress} if $\omega_{ij} > 0$ for a cable $\{i,j\}$, and $\omega_{ij} < 0$ for a strut $\{i,j\}$.  There is no condition for a bar.  The following is a basic duality result of Roth and Whiteley \cite{Roth-Whiteley}.

\begin{theorem}\label{thm:Roth-Whiteley} A tensegrity is infinitesimally rigid if and only if the underlying bar framework is infinitesimally rigid and there is strict proper equilibrium stress.
\end{theorem}

Since infinitesimal rigidity involves the solution to a system of linear equations and inequalities, we have certain relationships among the number of vertices of a tensegrity, say $n$, the number of edges, $e$, and the dimension of the space of equilibrium stresses $s$.  Recall that the dimension of the space of trivial infinitesimal flexes is $2$, given by translations in the torus.

\begin{proposition}\label{proposition:count} For an infinitesimally rigid tensegrity on a torus $\T^2$ with all struts, 
\[
e \ge 2n-1\,\,\,\,\, \text{and} \,\,\,\,\,s= e- (2n-2).
\]
\end{proposition}

\begin{corollary}\label{cor:rigid-count}If a packing ${{\bf P}}$ in a torus $\T^2$ is collectively jammed with $n$ disks and $k$ contacts, then $k \ge 2n-1$, and further when ${{\bf P}}$ is collectively jammed, it has exactly one stress if and only if $k=2n-1$. 
\end{corollary}

The idea is that there are $2n$ variables describing the configuration of the disk centers.  One packing disk can be pinned to eliminate trivial translations, and at least one extra constraint must be added to insure a self stress, $2(n-1) +1=2n-1$ constraints, corresponding to contacts altogether.

With isostatic packings, the stress space is one-dimensional, assuming the packing is rigid, not counting rattlers, if and only if $e=2n-1$.  In the granular material literature, this situation is called \emph{isostatic}.  However, in the mechanical engineering literature a bar tensegrity (framework) is called isostatic if it is infinitesimally rigid and it has no non-zero equilibrium stress, because when the framework is subjected to an external load, it can ``resolve" that load with a unique single internal stress.  So for packings, the isostatic conjecture is that when the packing disks are sufficiently generic, then there is only a one-dimensional equilibrium stress and $e=2n-1$.

Notice that the packing of the $6$ disks in Figure \ref{fig:rattler}, with the rattler missing, has $12 = 2\cdot 6$ edges and so is not isostatic as we have defined it above.  All the disks have the same radius, and so are not generic.  By contrast, the packing graph of $2$ disks in Figure \ref{fig:2-disks-slant} in a torus with a slanted lattice has $2\cdot 2 - 1= 3$ edges and when it is jammed, and so it is isostatic.  On the other hand, the packing in Figure  \ref{fig:2-disks-square} in the  torus defined by a rectangular lattice has $4$ contacts and is not isostatic even though the ratio of the radii are generic and the lattice has one free parameter in the space of rectangular lattices.  So in this case, the isostatic conjecture is false, if one insists on choosing that subset of the possible lattices.  The moral of this story is that all the lattice parameters should be included in the generic condition.
\begin{figure}[!htb]
    \centering
    \begin{minipage}{.5\textwidth}
        \centering
        \includegraphics[width=0.5\linewidth]{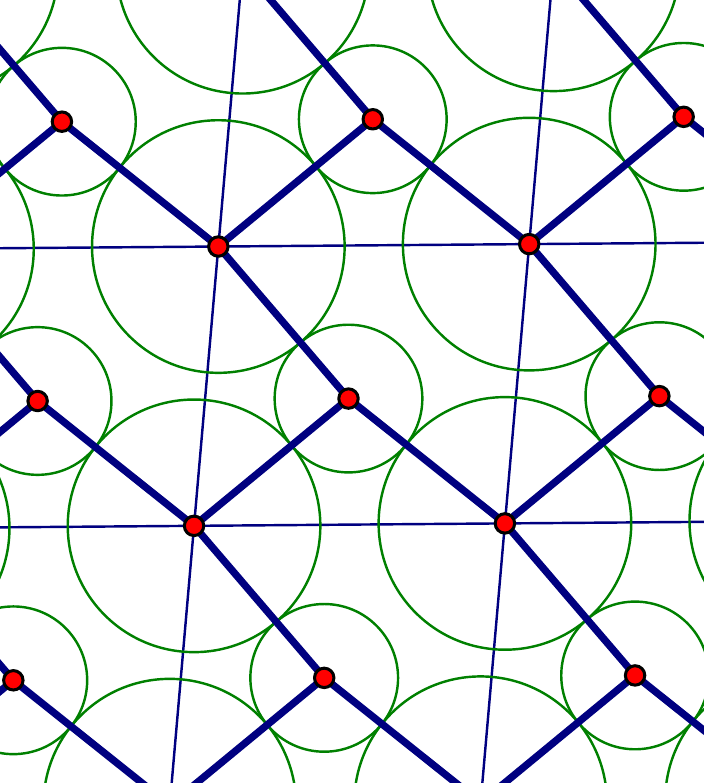}
        \caption{The slanted torus, with an \\ isostatic packing.}
        \label{fig:2-disks-slant}
    \end{minipage}%
    \begin{minipage}{0.5\textwidth}
        \centering
        \includegraphics[width=0.65\linewidth]{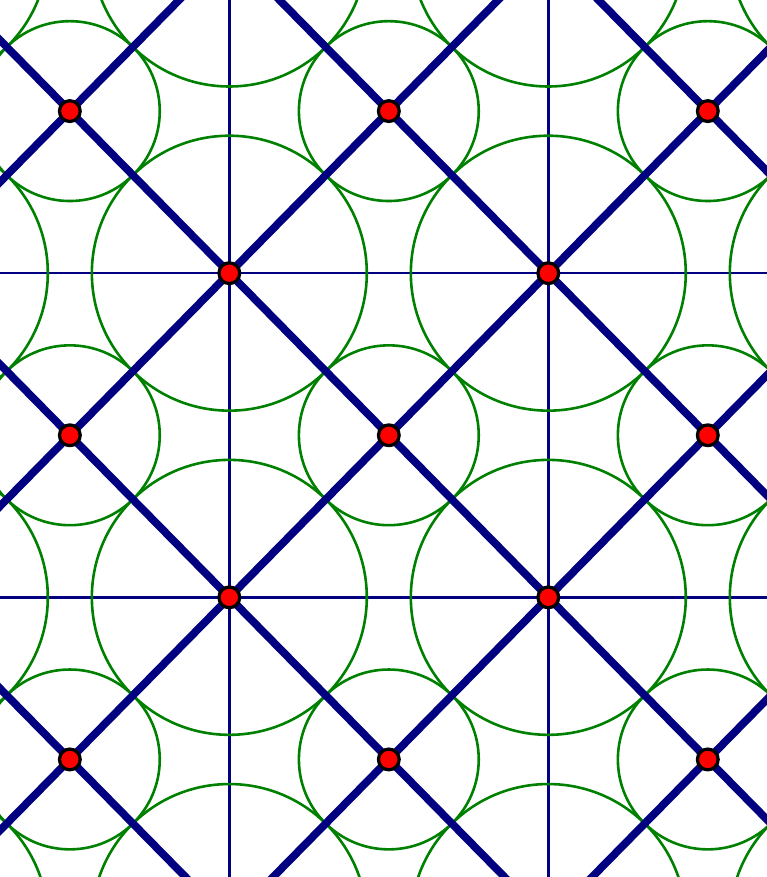}
        \caption{The rectangular torus, with a non-isostatic packing.}
        \label{fig:2-disks-square}
         \end{minipage}\\%
    \end{figure}

\section{Coordinates} \label{section:coordinates}

In order to do calculations later, we will describe the lattice, radii, and configuration in terms of canonical coordinates.

\begin{enumerate}[(a)]
	\item The lattice $\Lambda = \{ z_1 \lambda_1 + z_2 \lambda_2 \}$ where $\lambda_1=(a,0), \lambda_2=(b,c)$, and $z_1, z_2$ are integers $a,b,c >0$.  The dimension of $\mathcal L$ of all such lattices is $3$.  The dimension of such lattices with a fixed determinant, say $1$, is $2$.
	\item The configuration $\mathcal C$ is the set of $\pn$, where $\p_1=(0,0)$, and otherwise $\p_2, \dots, \p_n$ are free, and not constrained.  Note that a point $\p_i$ is defined to be equivalent to the point $\q_i$ if $\p_i +\Lambda = \q_i +\Lambda$ as sets.   The dimension of $\mathcal C$ is $2(n-1)=2n-2$.
	\item If the radius of the $i-$th disk is $r_i>0$, we denote the vector of radii as $\r=(r_1, \dots, r_n)$.  Later we will be interested in the relative ratios of these disk radii.  So we will denote the set of ratios as $\mathcal R = \{(1, r_2/r_1, \dots, r_n/r_1) \mid r_i >0 \}$. For any $\r=(r_1, \dots, r_n)$, define  $\bar{\r}= ( r_2/r_1, \dots, r_n/r_1)$.  So the dimension of $\mathcal R$ is $n-1$. 
\end{enumerate}

A packing ${{\bf P}}$ is described uniquely by all the coordinates above, where  $|\p_i-\p_j| = r_i+r_j$, for all $i,j$.  The space of all such packings will be denoted by ${\Pa}$.  If we fix the lattice $\Lambda$, then the corresponding space of such packings will be denoted by ${\Pa(\Lambda)}$.  Similarly if we additionally fix the radius ratios $\bar{\r}$, we denote that restricted packing space as ${\Pa(\Lambda,\bar{\r})}$.

\section{Dimension calculations} \label{section:dimension}

When we have a collectively jammed, or a locally maximally dense, packing, we would like to perturb the parameters, the radii and lattice, and still maintain that property.  In the following we will assume that the lattice $\Lambda$ is constant with determinant $1$.  So the packing is determined completely by the $n$ centers $\p$ of the packing disks and $\r$, the $n$ radii of the packing disks.  The pair $(\p, \r)$ determine a packing uniquely if and only if for all $i, j$ between $1$ and $n$ and $r_i >0$, 
\begin{equation}\label{eqn:packing}
|\p_i-\p_j| \ge r_i +r_j.
\end{equation}  
The density of $(\p, \r)$ is $\rho = \rho(\p, \r)= \pi \sum_i r_i^2=\rho(\r)$, which only depends on $\r$ and is clearly continuous.



\begin{proposition}\label{prop:continuity}  Suppose that $(\p, \r)$ is locally maximally dense in a torus given by the lattice $\Lambda$.  Let $\epsilon > 0$ be given.  Then there is a $\delta > 0$ such that for any packing $(\q, \s)$ such that $|\q-\p| < \delta$ and $|\s-\r| < \delta$, there is a locally jammed packing $(\q(1), \s(1))$ with  $|\q(1)-\p| < \epsilon$, and $\bar{\s}=\bar{\s(1)}$.
\end{proposition}

\begin{proof} Fix $\bar{\r}(0)$ and any configuration $\p(0)$ such that $(\p(0), \r(0))$ is a packing, i.e. it satisfies (\ref{eqn:packing}), and such that it is collectively jammed.  If there are any rattlers, we can deal with the collectively jammed subset, which we call the \emph{the spine}.  So we can assume that  $(\p(0), \r(0))$ is collectively jammed.  Let $\mathcal{P}$ be the space of packings of $n$ disks in the $\Lambda$ torus given by $(\p,\r)$, which corresponds to  $(\p,s,\bar{\r})$, where $s=r_1$.

We define a \emph{constraint space} $\mathcal{E}=((\dots l_{ij} \dots), (r_1, r_2,\dots, r_n))$, where $l_{ij}$ are positive real variables that correspond to edges of the contact graph on $n$ vertices, and $\r=(r_1, r_2,\dots, r_n)$ are also positive real variables that correspond to the disks in the packing.   We define a set $E \subset \mathcal{E}$ defined by the following constraints:

\begin{eqnarray}
l_{ij} &\ge& s(\bar{r}_i + \bar{r}_j)\label{packing-condition}\\
\bar{\r}&=&\bar{\r}(0)\label{radii-condition}\\
\rho(\r) &\ge& \rho(\r(0))\label{density-condition}
\end{eqnarray}

We then define a continuous map $f: \mathcal{P} \rightarrow \mathcal{E}$, by 
\[f(\p,\bar{\r},s) = (|\p_i-\p_j| \dots, \bar{\r}, r_1).
\]

Since $(\p(0), \r(0))$ is collectively jammed, there is compact neighborhood $C \subset \mathcal{P}$, the \emph{rigidity neighborhood} such that $f^{-1}({E})\cap C$ is just $\{(\p(0), r_1, \bar{\r}(0))\}$, that is the packing given by $(\p(0), \r(0))$.

Define the $\delta$ neighborhood of ${E}$, ${E}_{\delta}$ by the conditions

\begin{eqnarray}
l_{ij} &>& s(\bar{r}_i + \bar{r}_j)-\delta \label{near-packing-condition}\\
|\bar{\r}-\bar{\r}(0)| &<& \delta \label{near-radii-condition}\\
\rho(\r) &>& \rho(\r(0))-\delta \label{near-density-condition}
\end{eqnarray}

Note that Conditions (\ref{near-packing-condition}),   (\ref{near-radii-condition}), (\ref{near-density-condition})  correspond to a slackened versions of (\ref{packing-condition}), (\ref{radii-condition}), (\ref{density-condition}), respectively.  Then for every $\epsilon > 0$, there is a $\delta >0$ such that $f(\q,\bar{\r},s) \in {E}_{\delta}$ implies that $|(\p(0),\r(0))-(\q,\r)| < \epsilon$.  That is $f^{-1}({E}_{\delta})\cap C \subset U_{\epsilon}$, the $\epsilon$ neighborhood of  $(\p(0),\r(0))$. This is due to $C$ being compact and $E$ closed.  See \cite{Connelly-energy}, Theorem 1, for a similar argument.   

Next start with any packing $(\q,\r) \in f^{-1}({E}_{\delta})\cap C \subset U_{\epsilon}$ that maps to ${E}_{\delta}$ and continuously increase its density $\rho(\r)$ fixing $\bar{\r}$ until it reaches a local maximum, where it becomes locally maximally dense.  During this process, the packing will always satisfy (\ref{near-packing-condition}),   (\ref{near-radii-condition}), (\ref{near-density-condition}), and therefore be a local maximally dense packing $(\q(1),\r(1))$, remaining in  $U_{\epsilon}$.  See Figure \ref{fig:rigidity-nbd} for a visualization of this process.  \qed

 \begin{figure}[!htb]
    \centering
        \includegraphics[width=0.8\textwidth]{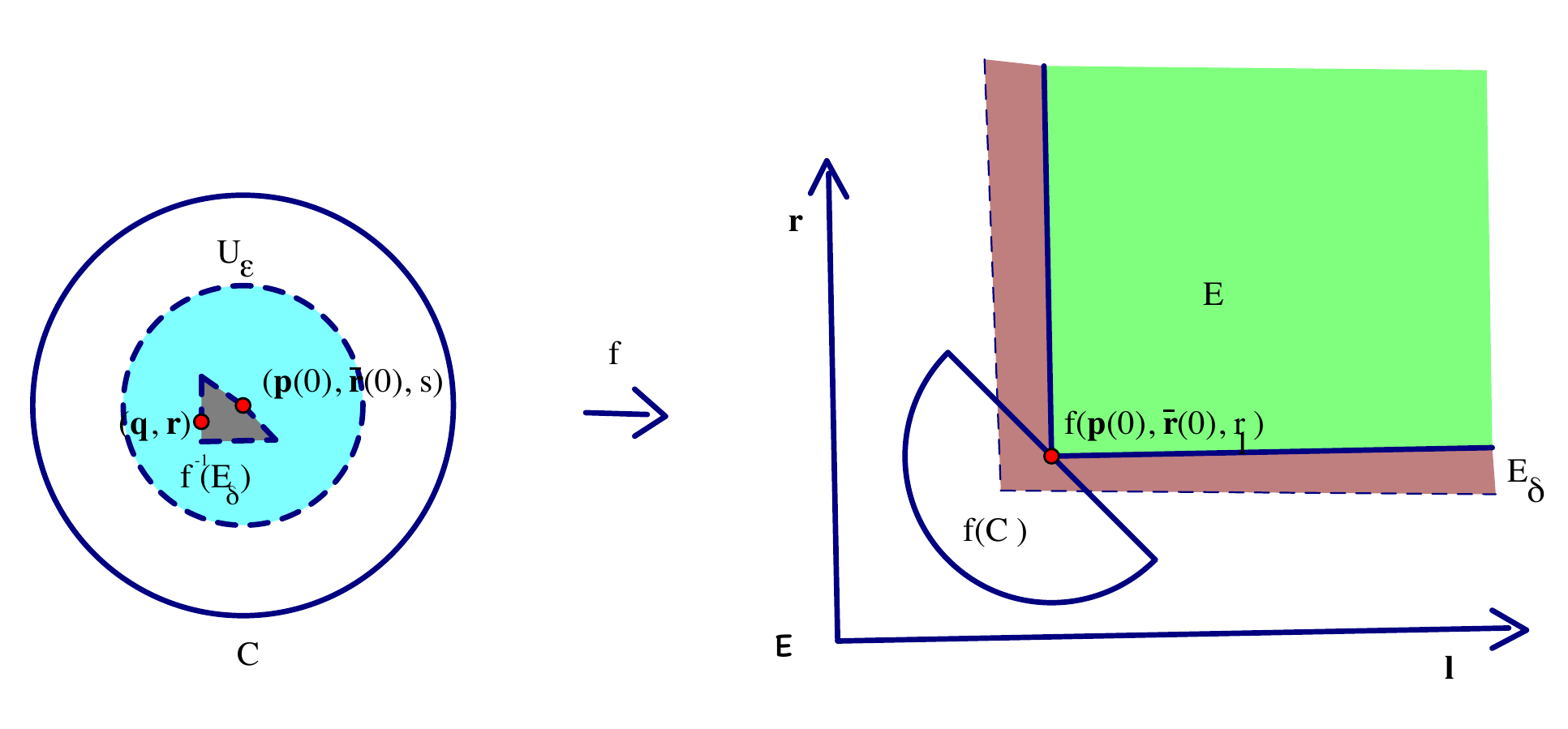}%
        \captionsetup{labelsep=colon,margin=2cm}
         \caption{This shows the argument in the proof of Proposition \ref{prop:continuity}.  The vertical direction in this figure also represents an increase in density $\rho$.}
\label{fig:rigidity-nbd}
    \end{figure}

\end{proof}

\begin{corollary} Under the same assumptions as Proposition \ref{prop:continuity}, for any radius ratio in  $\mathcal R$ within $\delta$ of $\bar{\r}$, there is a locally maximally dense packing $(\q, \s)$ such that  $\bar{\s}$ is that radius ratio, and $\q$ is within $\delta$ of $\p$.  Furthermore, the lattice can be perturbed by $\delta$ as well, with the same conclusion.
\end{corollary}

\begin{proof}Apply Proposition \ref{prop:continuity} to the packing $(\p, \r')$, where all the radii of $\r'$ are strictly smaller than those of $\r$ but $\bar{\r}'$ still close to $\bar{\r}$, which is a feasible packing, but such that the ratios are the given ratio in $\mathcal R$.  Similarly, one can alter the lattice by a sufficiently small amount and apply the same limiting argument to the altered lattice converging to $\Lambda$. \qed
\end{proof}
\bigskip

Let $X(n, CJ)$ be the space of collectively jammed packings with $n$ disks.  Note that this space is quantified over all lattices, and dimension of such lattices, by the definition $(a)$, is two, where packings are identified with the configuration $\p$, and radii $\r$.  
We assume that there is at least one collectively jammed packing.

\begin{corollary}\label{cor:CJ} The dimension of $X(n, CJ)$ is $n+1$.
\end{corollary}

Notice that the perturbed packing ${(\q,\s)}$ in the proof of Proposition \ref{prop:continuity} may loose some packing contacts from the original and even possibly create some rattlers as in Figures \ref{fig:tricusp-3} and \ref{fig:2-rattlers} in the tricusp case.  In that case, such packings, being not collectively jammed, are not in the set $X(n, CJ)$ and, being of lower dimension than $n+1$, do not contribute to the dimension of $X(n, CJ)$.  
Note that the area of the torus corresponding to the lattice $\Lambda$ is the determinant of the matrix defining $\Lambda$, which is just $ac$ from the definition in Section \ref{section:coordinates}.  With this generality we have the following.

\begin{theorem}\label{theorem:dimension} The dimension of the space of packing radii for locally maximally dense packings in a neighborhood of a fixed collectively jammed configuration with a lattice in the neighborhood of that fixed lattice, and radius ratios in the neighborhood of those fixed ratios, with $n$ disks is $n+1$, modulo rattlers.  
\end{theorem}

\begin{proof}  By Proposition \ref{prop:continuity}, for each configuration $\p$,  radius $\r$  and lattice $\Lambda$ that is collectively jammed, and therefore locally maximally dense,  there is a locally maximally dense packing $(\q,\s,\Lambda)$.    Each choice of $\Lambda$ and $\bar{\r}$ has a distinct locally maximally dense packing modulo rattlers.  There are $n-1$ choices for radius ratios and $2$ choices for lattices, $n+1$ in all.  \qed
\end{proof}

Here the rattlers are counted as not contributing to the dimension of the space of packings.  It is as if they were stuck to the rest of the packing.  If they were counted, then each rattler would add $2$ degrees of freedom to their configuration space.  If the given packing  has rattlers they will contribute the same degrees of freedom to each of the approximations, and they can be disregarded.  If the given packing  has no rattlers, it can happen that some of the approximation packings could themselves have rattlers.  But we will see that this cannot happen when the given packing is isostatic.  Next we suppose that the packing is approximated by another with the same graph.

\begin{proposition}\label{proposition:near-rigid} Let ${{\bf P}} \in \Pa$ be any collectively jammed packing with configuration $\p$,  radii $\r$, and lattice $\Lambda$.  Then there is an $\epsilon >0$ such that for any other packing ${{\bf Q}} \in \Pa$ where $|{\bf P}-{\bf Q}|<\epsilon$, and the packing graph of ${{\bf P}}$ is the same as the packing graph of ${{\bf Q}}$, then  ${{\bf Q}}$ is collectively jammed as well.
\end{proposition}

\begin{proof}By Theorem \ref{theorem:Danzer}, ${{\bf P}}$ is collectively jammed if and only if it is infinitesimally rigid.  As before, if there is no $\epsilon$ as in the statement, there is a sequence of 
${{\bf Q}}_j, \,j= 1, 2, \dots$ converging to ${{\bf P}}$, each with its own configuration $\q(j)$ and infinitesimal flex $\q'(j)$ satisfying the infinitesimal rigidity constraint (\ref{eqn:inf-flex}).  By renormalizing we can assume that $|\q'(j)|=1$.  So as the $\q(j)$ converge to $\p$, and by taking a subsequence, the $\q'(j)$ converge to a non-zero (and thus non-trivial by the conventions in Section \ref{section:coordinates}) infinitesimal flex of $\p$.  Thus there is such an epsilon as in the conclusion.  \qed
\end{proof}

Suppose that a packing has contact graph $G$, and define $X(n, CJ(G)) \subset X(n, CJ)$ as the set of collectively jammed packings with the given contact graph $G$.  

\begin{corollary}\label{CJ:raterless} In a sufficiently small neighborhood of a packing ${{\bf P}}$ with contact graph $G$, suppose that no contacts are lost in the space of collectively jammed packings.  Then the dimension of $X(n, CJ(G))$ is $n+1$
\end{corollary}

Notice that the statements here do not depend on the two dimensional analytic theory that we describe in the next Section \ref{section:analytic}, and indeed there are higher dimensional statements that we will not go into detail here.  Notice, also, that the results in this section hold in any higher dimension with appropriate adjustment for the dimension of the space of lattices in Theorem \ref{theorem:dimension}.

\section{Analytic theory of circle packings} \label{section:analytic}

We need to first do some bookkeeping as far as the topology of graphs on the surface of $\T^2$.  One of the first results of Andreev and Koebe was to start with a triangulation of a surface, say a torus, and then create a circle packing whose graph is that triangulation.  This is explained in careful detail in Stephenson's book \cite{Stephenson}.  Note that this pays no attention to the radii of the packing.  So suppose that there are $n$ circles in a triangulated packing, with $e_T$ edges, and $T$ triangles.  Since each edge is adjacent to $2$ triangles and each triangle is adjacent to $3$ edges, we have 
\[
3T=2e_T,
\]
and since the Euler characteristic of $\T^2$ is $0$, we get that 
\[
n - e_T + \frac{2}{3}e_T=0,
\]
\begin{center}and $e_T= 3n$.\end{center}
\bigskip
In our case, we usually do not have a complete triangulation of the torus.  Indeed, from Section \ref{section:rigidity}, we only have $e=2n-1$ edges if the packing is isostatic.  But our packing graph is embedded in $\T^2$, and it can be completed to a triangulation by adding $e_T-e=n+1$ additional edges.
 But this does not help us since we do not want the extra contacts of a triangulation.  For example, in Figure \ref{fig:n3} we see an isostatic packing of $3$ disks in a square torus with $2\cdot 3 -1=5$ edges, where $3+1=4$ additional edges have been inserted to create a triangulation of the torus.  Notice that there are multiple edges between some pairs of vertices, but in the universal cover, as shown, the edges form an actual triangulation of the plane.
 
 \begin{figure}[htbp]
    \centering
        \includegraphics[width=0.4\textwidth]{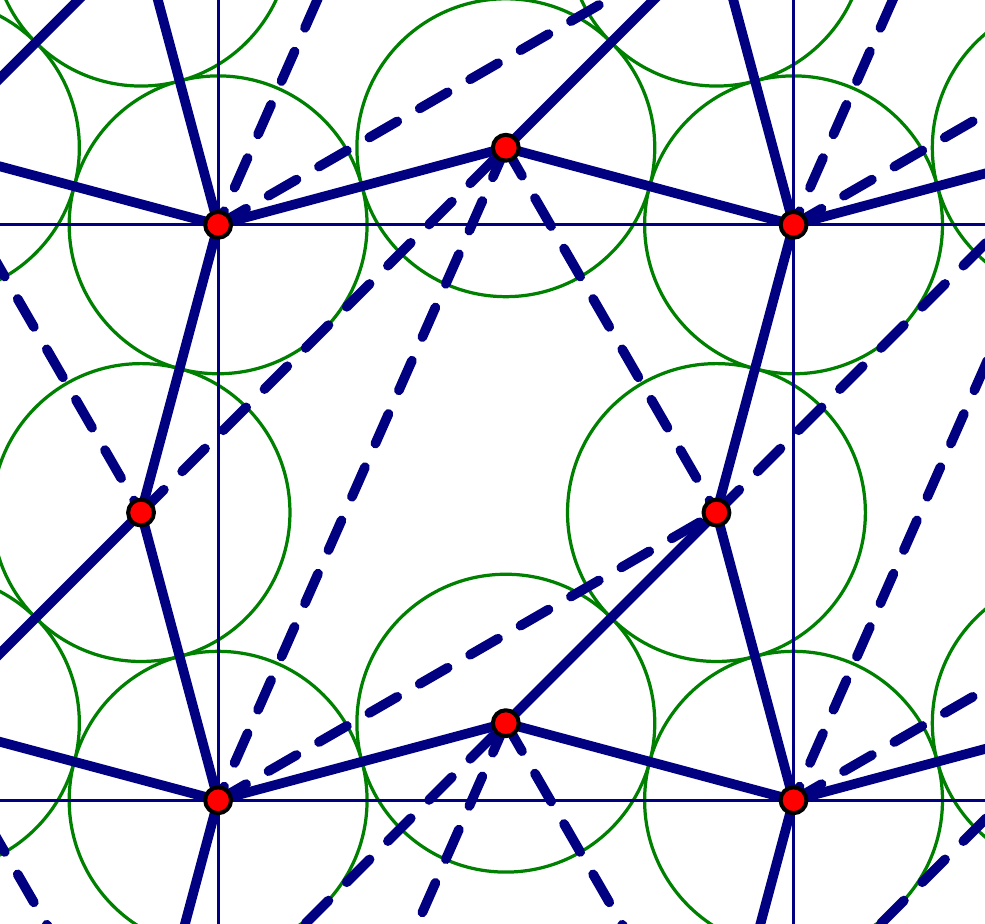}%
         \captionsetup{labelsep=colon,margin=2cm}
    \caption{This is a maximally dense packing of $3$ equal disks in a square torus \cite{Will-square}  with a minimal number of contacts,  namely $5$, and so it is isostatic. We have inserted $4$ additional dashed edges to create a triangulation with $9$ edges total.}
    \label{fig:n3}
    \end{figure}
    
There has been a lot of interest in circle arrangements,  where some pairs of circles are forced to intersect at certain specified angles, generalizing the  Koebe-Andreev-Thurston result, where they are made to be tangent.  At the other end of that construction, there is a way to measure the distance between pairs of non-tangent circles with disjoint interiors. The \emph{inversive distance} between two circles is defined as 
 \[
 \sigma(D_1,D_2)=\frac{|\p_1-\p_2|^2-(r_1^2+r_2^2)}{2r_1r_2},
 \]
where $D_1$ and $D_2$ are disks with corresponding radii $r_1$ and $r_2$.  It does not seem that there is a proof known, where any set of inversive distances determine a configuration with those inversive distances.  But the following local result by Ren Guo \cite{Guo} is enough for our purposes.

\begin{theorem}\label{theorem:Guo} Let $\mathcal T$ be a triangulation of a torus corresponding to a circle packing ${{\bf P}}$ where $\sigma(D_i,D_j) \ge0$ is the inversive distance between each pair of disks $i, j$ that are an edge in the triangulation  $\mathcal T$.   Then the inversive distance packings are locally determined by the values of $\sigma(D_i,D_j)$.
\end{theorem}

\begin{corollary}\label{cor:Guo}
The dimension of $X(n, CJ(G))$ is no greater than $3n-k$, where the number of contacts in the graph $G$ is $k$.
\end{corollary}
\begin{proof}  By Theorem \ref{theorem:Guo} the dimension of all packings with $G$ as the contact graph is no greater than $3n-k$ and $X(n, CJ(G))$ is a subset of those packings. \qed

\end{proof} \qed

By varying the values of the inversive distances, it is possible to show that in the neighborhood of a collectively jammed packing, a configuration that has those particular inversive distances and that the dimension of $X(n, CJ(G))$ is exactly $3n-k$.  We expect to show this in a later work. 
One should keep in mind, though, that the metric of the ambient space, which in our case is the lattice determining the torus,  may change with the deformations of the inversive distance data, and this should be taken into account when doing our dimension calculations.

\section{The generic property} \label{section:generic}

If one has a collection of real numbers $X$, they are defined to be \emph{generic} if  there are no solutions to non-zero polynomials $p(x_1, \dots, x_n)=0$ with integer coefficients, where each $x_i \in X$.  Each such polynomial defines an algebraic set and being generic implies that the points of $X$ avoid that set.  But this is something of an overkill.  For example, in many cases, for some given set $X$ and situation at hand, there are only some finite number of such algebraic sets that have to be avoided, but it may be difficult to explicitly define what the particular polynomials are.  

Another way to think of the generic parameters $x_1, \dots, x_n$ is as independent variables satisfying no polynomial relations over the integers (or equivalently the rationals).  In our case, the independent parameters are the ratios of distinct radii of the packing disks.


\section{The isostatic theorem} \label{section:main}

\begin{theorem}\label{theorem:isostatic} If a collectively jammed packing ${{\bf P}}$ with $n$ vertices in a torus $\T^2=\R^2/\Lambda$ is chosen so that the ratio of packing disks $\bar{\r}$, and torus lattice $\Lambda$, is generic, then the number of contacts in ${{\bf P}}$ is $2n-1$, and the packing graph is isostatic.
\end{theorem}

\begin{proof} Restrict to a sufficiently small neighborhood of a collectively jammed packing ${{\bf P}}$.  Any restriction on the number of contacts in the space $X(n,CJ)$ of collectively jammed packings constitutes an algebraic or semi-algebraic subset, and corresponds to a constraint in the $\bar{\r}$ and lattice variables unless it defines (an open subset of) the whole space $X(n,CJ)$.  This is because there is a natural projection from collectively jammed  packings $(\p, \r, \Lambda)$ to $(\bar{\r}, \Lambda)$, where $\Lambda$ has determinant $1$, the dimension of  $X(n,CJ)$ in $n+1$, and the dimension of the $(\bar{\r}, \Lambda)$ is also $n+1$.  Therefore for generic $(\bar{\r}, \Lambda)$, where $\Lambda$ has determinant $1$, we may assume the number of contacts in the neighborhood of  ${{\bf P}}$ is constant.   By Corollaries \ref{cor:Guo},  \ref{CJ:raterless}, \ref{cor:CJ}, for packings with $n$ vertices and $k$ contacts that locally have a constant number of contacts among collectively jammed packings and thus have constant graph $G$, the following holds: 

\[3n-k \ge \dim(X(n, CJ(G)))  = \dim(X(n, CJ))= n+1.
\]
Thus $3n -(n+1)=2n-1 \ge k$.  By Corollary \ref{cor:rigid-count}  such packings are isostatic.  If the radii ratios $\bar{\r}$ and the lattice variables are generic, and the packings in the neighborhood of ${{\bf P}}$ are collectively jammed with a constant number of contacts, then those packings are isostatic, as was to be shown.  \qed
\end{proof}

One could do a similar calculation for the case when each packing is assumed to be strictly jammed, allowing the lattice to vary.  Then the number of contacts in this case is $2n+1$ instead of $2n-1$.

It is also possible to extend Theorem \ref{theorem:isostatic} to cases when the number of free variables used to define a generic $\bar{\r}$ and lattice are less than $n+1$, where $n$ is the number of disks.  For example, suppose the number of free lattice variables is one instead of two. We can show, if the $\bar{\r}$ and lattice variables are otherwise generic,  then the number of contacts is at most $2n$ instead of $2n-1$.  An example of this is shown in Figure  \ref{fig:2-disks-square}, where $n=2$, the number of free lattice variables is one instead of $2$ and there is just one radius ratio.  The number of contacts in a generic case is $4$ instead of 3.  This is the only example we know so far though.

\section{Computations} \label{section:main}

Will Dickinson, et al. in \cite{Will-square} showed that the most dense packing of $5$ equal circles in a square torus is when they form a grid as in Figure \ref{fig:5-square}.  This packing is not isostatic since it clearly does not have just a one-dimensional equilibrium stress, but two.   It has $2\cdot 5=10$ contacts, one more than needed for rigidity.  When the disk radii are perturbed, we get the packing in Figure \ref{fig:5-disks-2}, which is isostatic with $9$ contacts even though the defining lattice is still the square lattice.
\begin{figure}[!htb]
    \centering
    \begin{minipage}{.5\textwidth}
        \centering
        \includegraphics[width=0.8\linewidth]{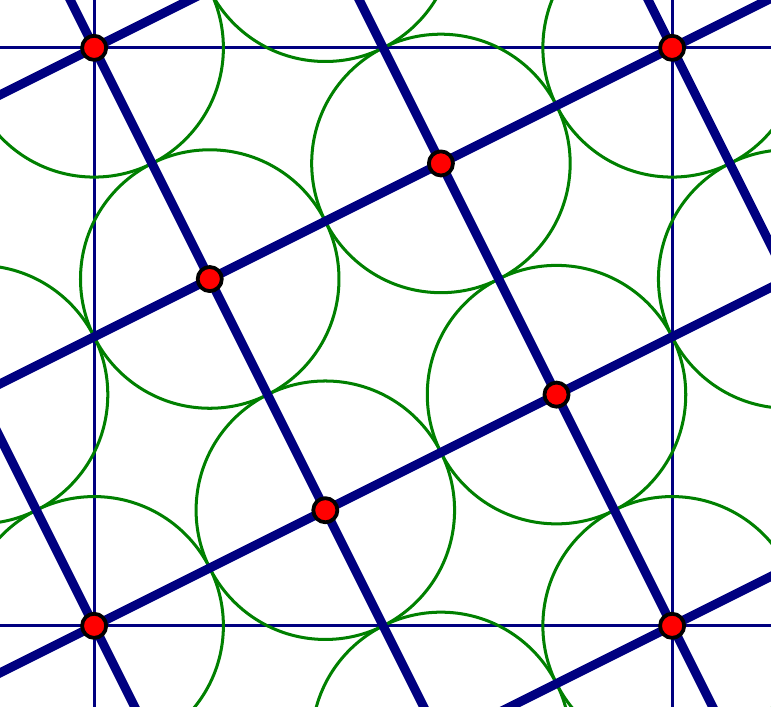}
        \caption{The most dense packing of $5$ disks\\ in the square torus.}
        \label{fig:5-square}
    \end{minipage}%
    \begin{minipage}{0.475\textwidth}
        \centering
        \includegraphics[width=0.8\linewidth]{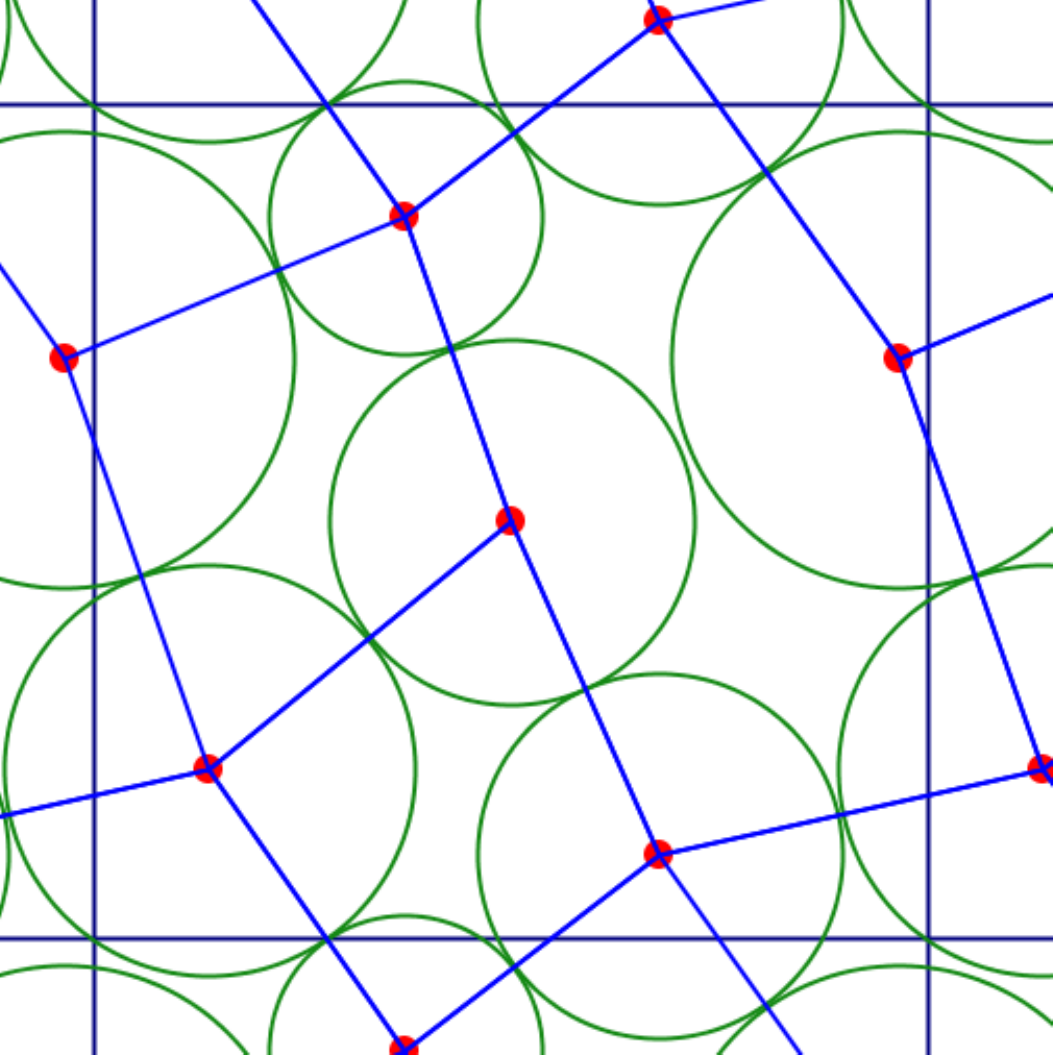}
        \caption{An isostatic packing of  $5$  disks with generic radii in the square torus.}
        \label{fig:5-disks-2}
         \end{minipage}\\%
    \end{figure}
\begin{figure}[!htb]
    \centering
    \begin{minipage}{.555\textwidth}
        \centering
        \includegraphics[width=0.8\linewidth]{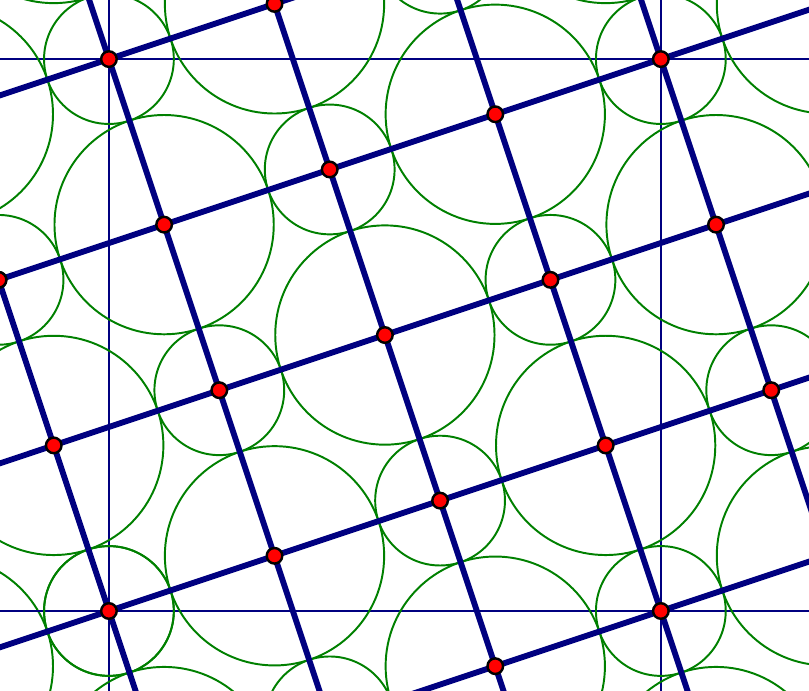}
        \caption{A packing of $10$ disks in the square\\ torus.}
        \label{fig:10-unequal-disks}
    \end{minipage}%
    \begin{minipage}{0.475\textwidth}
        \centering
        \includegraphics[width=0.8\linewidth]{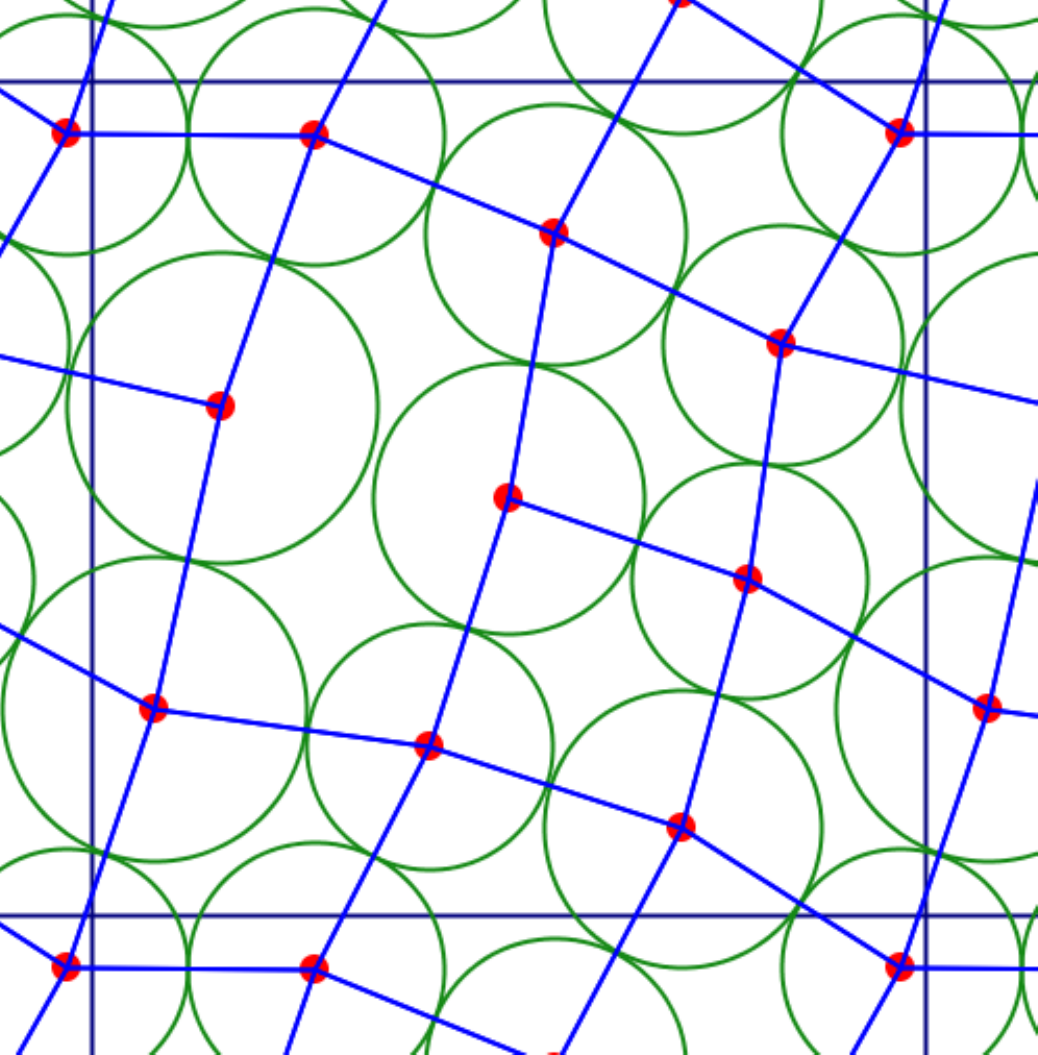}
        \caption{An isostatic packing of  $10$  disks\\ with generic radii in the square torus.}
        \label{fig:10-disks}
         \end{minipage}\\%
    \end{figure}
    
 Similarly, for $10$ disks in a square torus, even when there are two different radii, a rigid jammed packing may not be isostatic as in Figure \ref{fig:10-unequal-disks}.  When the radii are sufficiently varied, Figure \ref{fig:10-disks} shows how one of the contacts breaks, obtaining another isostatic packing.  Note that in these packings, we still get an isostatic packing without having to perturb the the underlying square lattice.   Indeed, it is tempting to propose that if the packing has a sufficiently large number of packing elements, then it will be isostatic or at least have the minimum number of contacts for the collectively jammed case (or the strictly jammed case) even if all the disks have the same radius.  This does not seem to be the case with some of our calculations, and for calculations done by Atkinson et al in \cite{Torquato-mono}.   
 
The packings of Figure \ref{fig:5-disks-2} and Figure \ref{fig:10-disks} were obtained with a ``Monte Carlo" algorithm similar to the one described in \cite{Donev-Torquato-Connelly} and \cite{Torquato-tricusp}, where a seed packing with generic radii with no contacts allows the radii to grow until the packing jams.  This generally works when there is enough random motion to force the packing to be rigid.  If the packing seems not to be converging to an infinitesimally rigid configuration, it is always possible to apply a linear programming algorithm to break up any configuration that is not converging sufficiently rapidly, as was described in \cite{Donev-Torquato-Connelly}.

\section{The tricusp case} \label{section:triceps}

There are many circumstances where there is a jammed packing in a bounded container with an appropriate condition on the boundary of the container, and the infinitesimal rigidity condition holds.  (The condition is that the boundary of the container must consist of concave up curves like the tricusp in Figure \ref{fig:tricusp-3}.)  It seems reasonable that if the shape of the container is generic, including the ratio of the radii, that the packing is isostatic.  If the container consists of three mutually tangent circles, then the isostatic conjecture will hold fixing the boundary, since the packing is determined up to linear fractional conformal transformations, and the three boundary circles can be fixed.   We call the region between the three mutually tangent circles a \emph{tricusp} following \cite{Torquato-tricusp}, as in  Figure \ref{fig:2-rattlers} and Figure \ref{fig:tricusp-3}.  If we have a jammed packing of $n$ disks in the tricusp, we regard the boundary as fixed and since there are no trivial motions,  there are $2n$ degrees of freedom for the centers of the disks. If the packing is isostatic, there is one other constraint due to the stress condition as before.  Thus there are exactly $e=2n+1$ contacts, or equivalently edges in the packing graph, when the packing is isostatic.

\begin{figure}[!htb]
    \centering
    \begin{minipage}{.5\textwidth}
        \centering
        \includegraphics[width=0.8\linewidth]{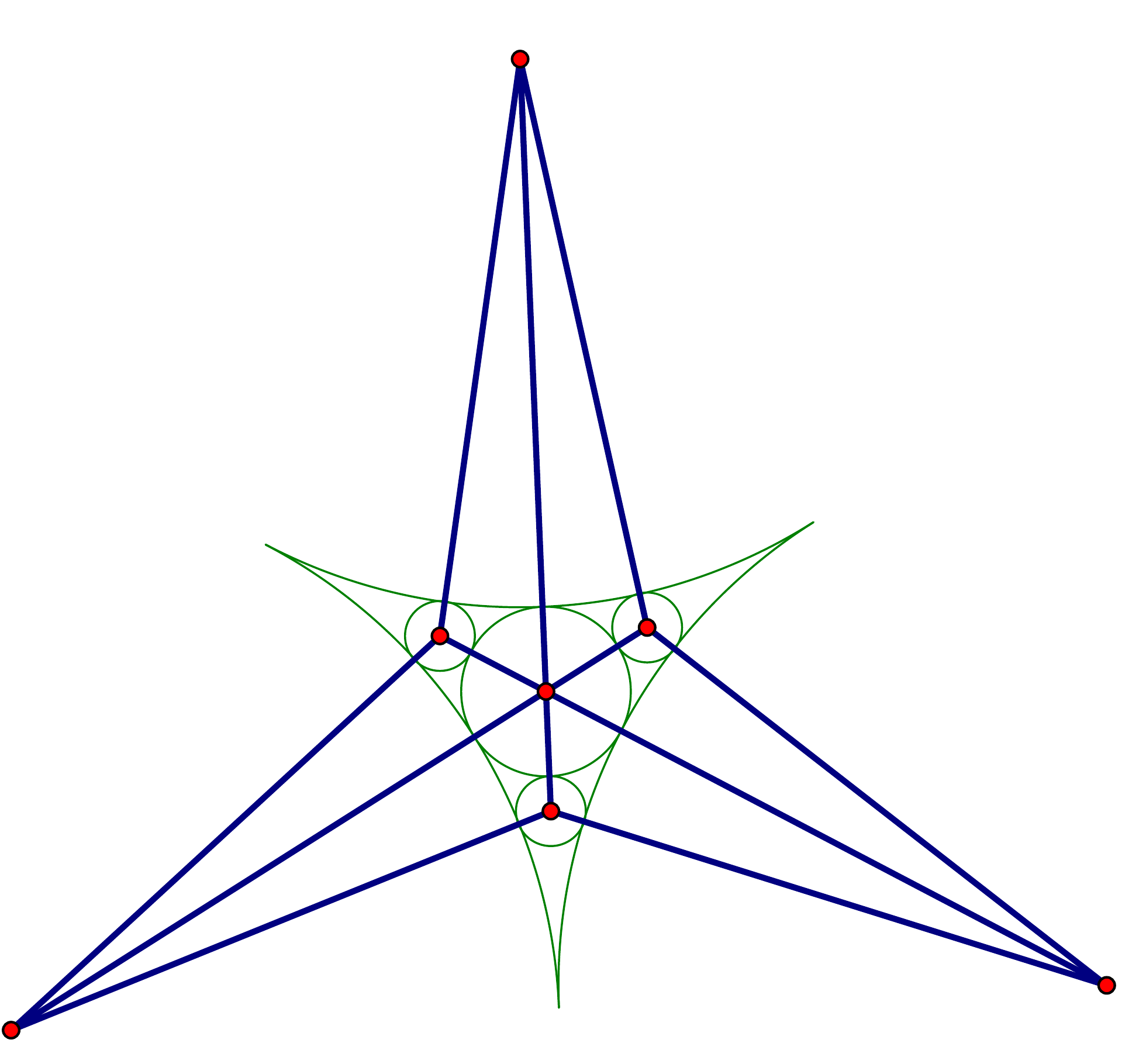}
        \caption{The tricusp with $4$ jammed disks, \\not isostatic.}
        \label{fig:tricusp-3}
    \end{minipage}%
    \begin{minipage}{0.5\textwidth}
        \centering
        \includegraphics[width=0.8\linewidth]{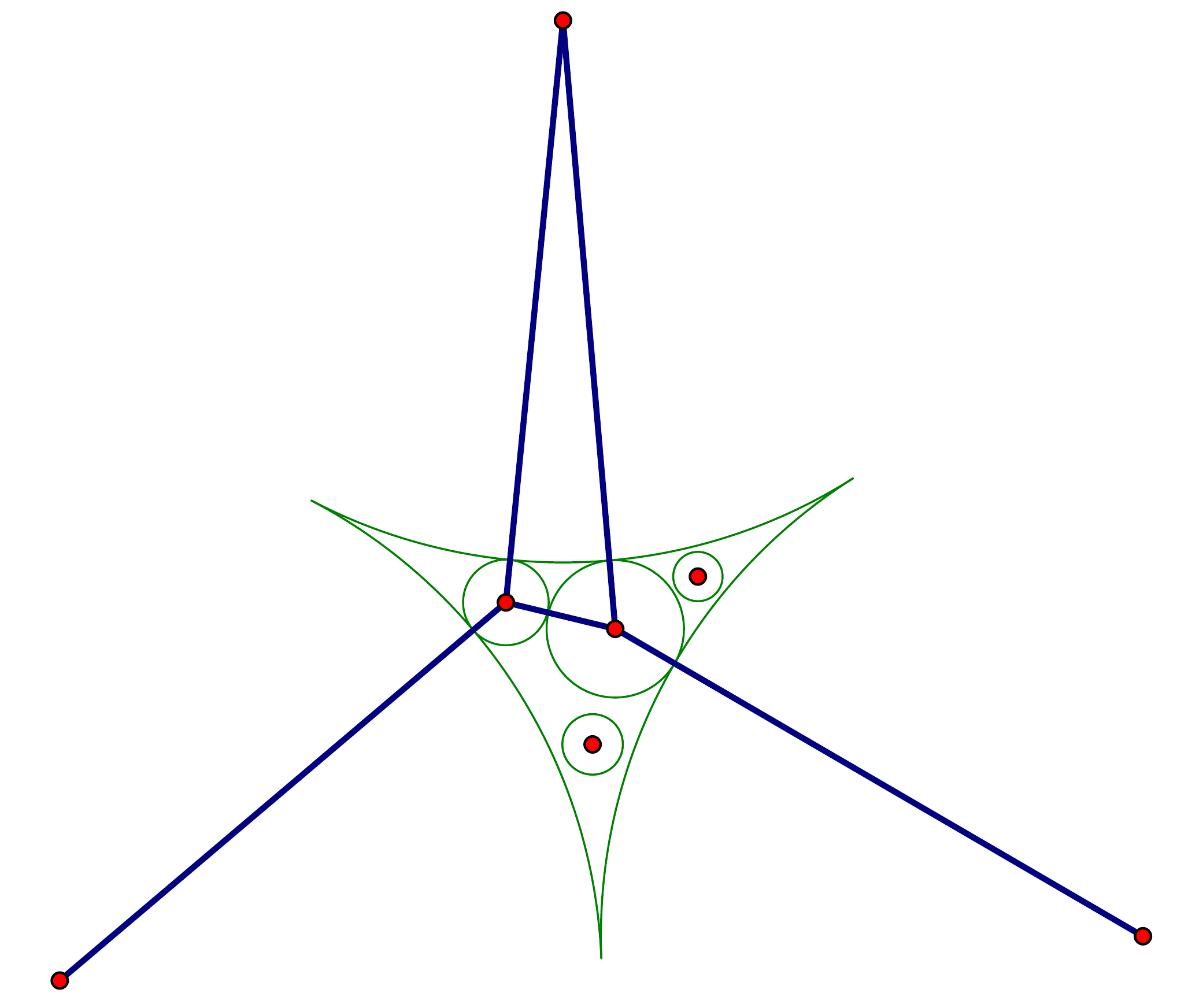}
        \caption{The tricusp with $2$ rattlers, isostatic.}
        \label{fig:2-rattlers}
         \end{minipage}\\%
    \end{figure}

On the other hand, if $e_T$ is the number of edges in a triangulation of a triangle, then $e_T= 3(n+3)-6=3n+3$.  This is because adding $3$ edges connecting the fixed edges, we get a triangulation of sphere.  Such a triangulation is well-known to have $3m-6$ edges when there are $m$ vertices, assuming that the graph is $3$-connected.   So there are $e_T-(e+3)= 3n+3-(2n+1+3)=n-1$ diagonal edges which can serve as free parameters for the inversive distance as we did for the case of the torus.  So we get the following.
\begin{theorem}If a jammed packing ${{\bf P}}_0$ with $n$ disks in a tricusp is chosen so that the ratio of the radii of the  packing disks is generic, then the number of contacts in ${{\bf P}}_0$ is $2n+1$, and the packing graph is isostatic. 
\end{theorem}

The packing in Figure \ref{fig:tricusp-3} is jammed, but when the ratios of the generic radii are perturbed so that the smaller $3$ disks are smaller than they are in Figure \ref{fig:tricusp-3}, while the larger disk, in ratio to the smaller disks, is larger than in Figure \ref{fig:tricusp-3}, we get packing with two rattlers as in Figure \ref{fig:2-rattlers}.  In this process, there can be no extra edges created in the packing graph.  However, Figure \ref{fig:tricusp-3} has $12$ edges with $4$ vertices in the tricusp, which is $3$ more edges than is needed for being isostatic.  Due to the way the radii are perturbed, two of the smaller disks must become rattlers, so $n$, the number of disks, is decreased by $2$, and the number of contacts is decreased by $7$, bringing the edge count to the isostatic case.  

\section{Varying the radii and lattice} \label{section:varying}

Instead of fixing the lattice that defines the torus, one can allow the lattice to vary as well as the individual packing disks.  In this process, if the radii of the disks (or the ratio of the radii) is fixed there is a result analogous to Theorem \ref{theorem:Danzer}.  If there is an infinitesimal motion of the lattice and disk centers, that satisfies the packing requirement (i.e.  the strut requirements on the edge lengths) and that satisfies the constraints on the lattice, (i.e. there is no increase in the area of the whole torus), then there is an actual motion that increases the overall density.  This is the following result from \cite{basics}.

\begin{theorem}If $\Lambda'$ and $\p'$ represent an infinitesimal motion of a lattice $\Lambda$ and its configuration $\p$ that determines a non-positive area change, then there is a smooth motion of the lattice with its configuration that strictly increases adjacent distances and decreases the volume unless $\p_i' = \p'_j$ for adjacent disks $i$ and $j$, and $\Lambda'$ is trivial.
\end{theorem}

The effect of this is to increase the density of the configuration while varying both the lattice and the configuration, but keeping the radii  (ratios) constant.  When the packing is locally maximally dense with these kinds of deformations, in \cite{Donev-Torquato-Connelly} the packing is called \emph{strictly jammed}.  Notice, also, that the minimum number of contacts for a strictly jammed packing is at least two more than the minimum when the configuration has its lattice fixed.  Namely for $n$ disks, there should be at least $n+1$ contacts.   As an example of the this kind of deformation, if we start with square lattice configuration of two equal disks, perform this deformation and end up with the most dense configuration of equal disks in the plane, where each disk is surrounded by six others as in Figure \ref{fig:2-square-torus} deforming to Figure \ref{fig:triangular}.  One can see how the fundamental region has changed shape from a square with $4$ contacts and density $\pi/4=0.78539\dots$ to a rectangle with $6$ contacts with density $\pi/\sqrt{12}=0.90689\dots$. 
\begin{figure}[!htb]
    \centering
    \begin{minipage}{.5\textwidth}
        \centering
        \includegraphics[width=0.8\linewidth]{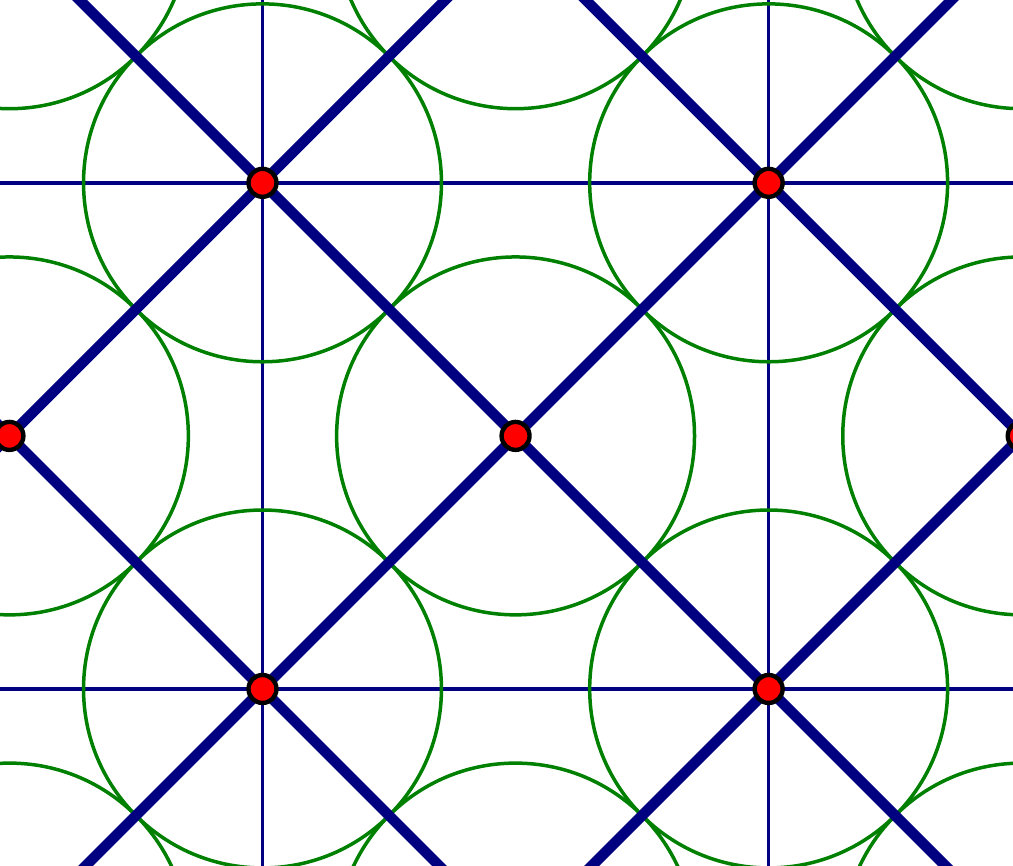}
        \caption{The most dense packing of two\\ disks in the square torus.}
        \label{fig:2-square-torus}
    \end{minipage}%
    \begin{minipage}{0.31\textwidth}
        \centering
        \includegraphics[width=0.8\linewidth]{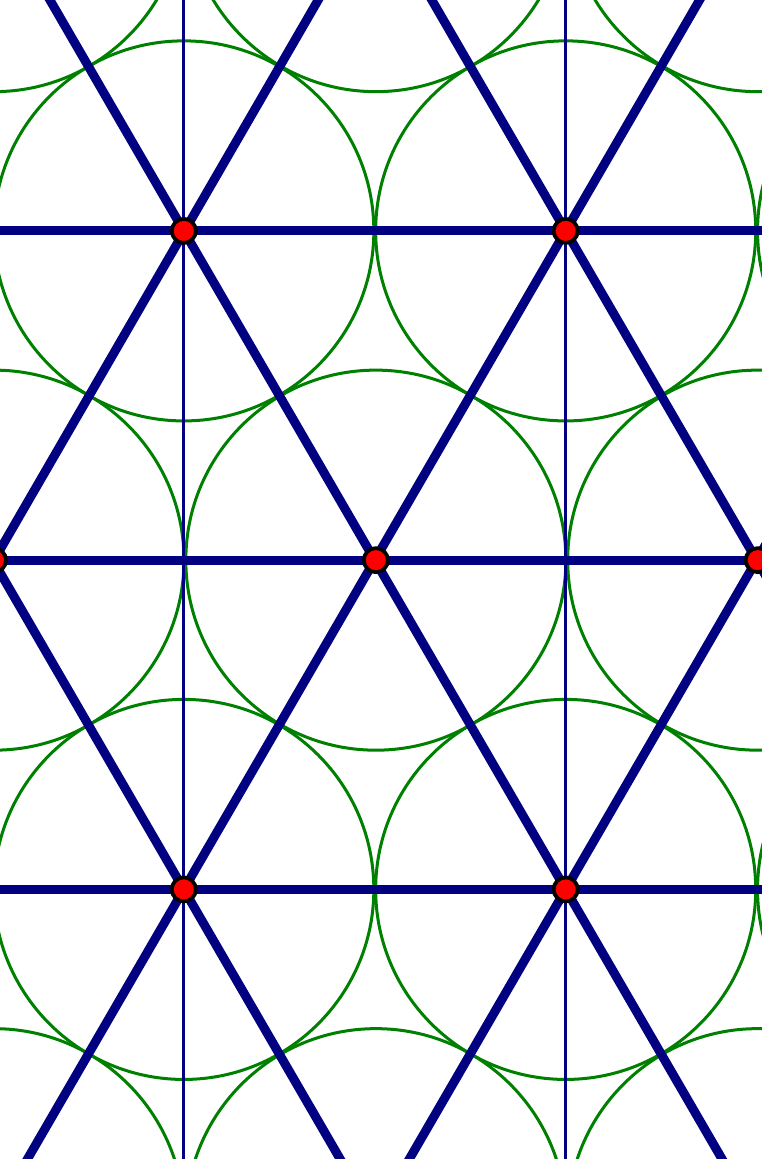}
        \caption{The most dense packing of equal disks in the plane.}
        \label{fig:triangular}
         \end{minipage}\\%
    \end{figure}
\begin{figure}[!htb]
    \centering
    \begin{minipage}{.41\textwidth}
        \centering
        \includegraphics[width=0.8\linewidth]{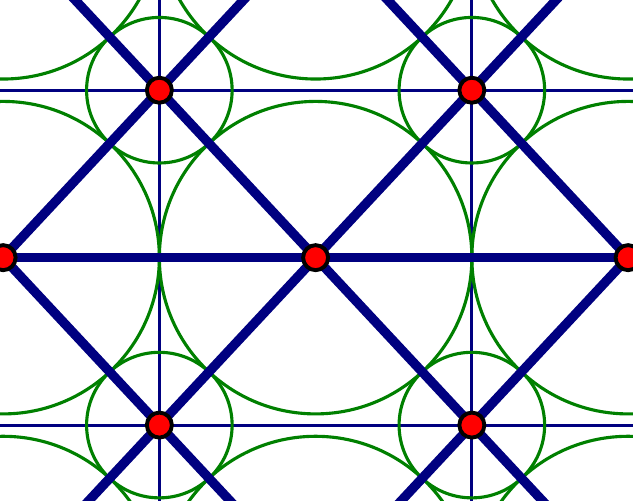}
        \caption{The conjectured most \\dense packing of two disks of  fixed \\ratio close to $\sqrt{2}-1$ in the plane.}
        \label{fig:binary-near-square}
    \end{minipage}%
    \begin{minipage}{0.39\textwidth} 
        \centering
        \includegraphics[width=0.8\linewidth]{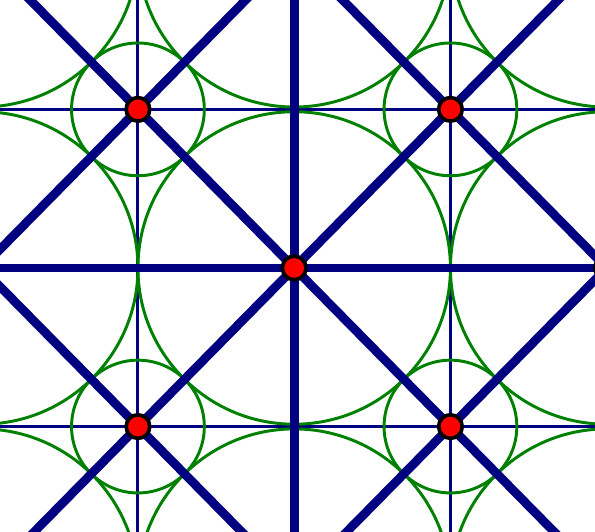}
        \caption{A most dense packing of disks with radius ratio $\sqrt{2}-1$ in the plane.}
        \label{fig:best-binary}
         \end{minipage}\\%
    \end{figure}
    
Another process one can use is to deform a packing, varying not only the configuration of centers, but varying the radii as well, while maintaing the condition that the disk centers with the determined radii form a packing.  There is similar second-order calculation for such deformations.  For example, for a packing with two disks in a torus having two different  radii, one starts again with the packing of Figure \ref{fig:2-square-torus} and deforms the lattice, the configuration of centers and radii as in Figure \ref{fig:binary-near-square} until eventually there is another contact  as in Figure \ref{fig:best-binary}.  Indeed, for any infinite packing of disks whose radii are $1$ and $\sqrt{2}-1=0.41421\dots$, in  \cite{Heppes-2-size-densest} Alid\'ar Heppes proved  that the maximum density for such a packing is $\pi(2-\sqrt{2})/2=0.92015\dots$, which is achieved by the packing in Figure  \ref{fig:best-binary}.  In the classic book by L\'aszl\'o Fejes T\'oth \cite{Fejes-Toth-book}, it is essentially conjectured that the packing in Figure \ref{fig:binary-near-square} has the maximum density for disks with two radii whose ratio is slightly greater than $\sqrt{2}-1$. 
The density of the the packings during the deformation from Figure \ref{fig:2-square-torus} to Figure \ref{fig:binary-near-square} is $\rho=\pi(1+r^2)/(4\sqrt{r^2+2r})$ as shown in Figure \ref{fig:density-plot}.  Note that the graph is concave up as predicted.  It known by Gerd Blind \cite{Blind} that for $0.742 < r < 1$, the most dense packing can do no better than $\pi/\sqrt{12}$ which is achieved by the ordinary triangular lattice packing, with just one disk size.
 \begin{figure}[h]
    \begin{center}
        \includegraphics[width=0.7\textwidth]{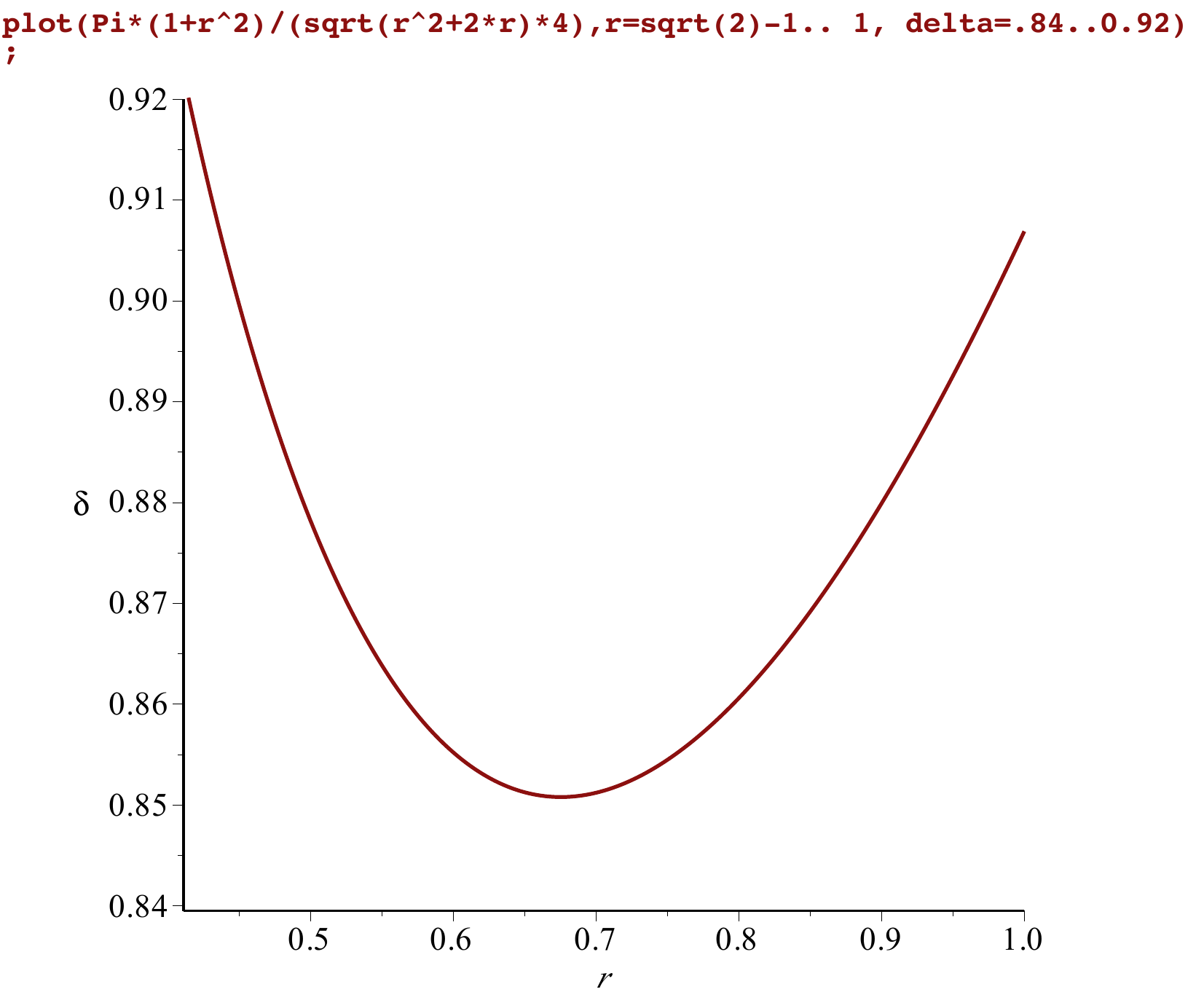}%
        \end{center}
        \captionsetup{labelsep=colon,margin=2cm}
    \caption{This plots the density of the packings from Figure \ref{fig:2-square-torus} to Figure\ref{fig:best-binary}.}
    \label{fig:density-plot}
    \end{figure}
     
From this discussion it seems that certain particular singular cases, with the most number of contacts, represent the most dense packings for two radii.  Away from those cases, the most dense packings may be when the singular cases are perturbed in a particular way.  
In \cite{Fejes-Toth-book} those singular packings were called \emph{compact} packings, which was defined to be when each packing disk is adjacent to and is surrounded by a cycle of packing disks, each touching the next.  But this is the same as saying that the graph of the packing is a triangulation.  

From the analytic packing point of view, one is given an abstract triangulation of a particular compact $2$-manifold.  Then the basic Koebe-Andreev-Thurston algorithm finds a circle packing with that given contact graph in a manifold of  constant curvature,  and this packing is unique up to the circle preserving linear fractional transformations of the manifold.  On the one hand this algorithm has no constraint that preserves the sizes of the radii.  On the other hand, if there are few enough of the disks as in Figures \ref{fig:triangular} and \ref{fig:best-binary}, or they are symmetric enough as in Figure \ref{fig:Kennedy-5}, then the Koebe-Andreev-Thurston algorithm will automatically have just two disk sizes.
\begin{figure}[h]
    \centering
        \includegraphics[width=0.6\textwidth]{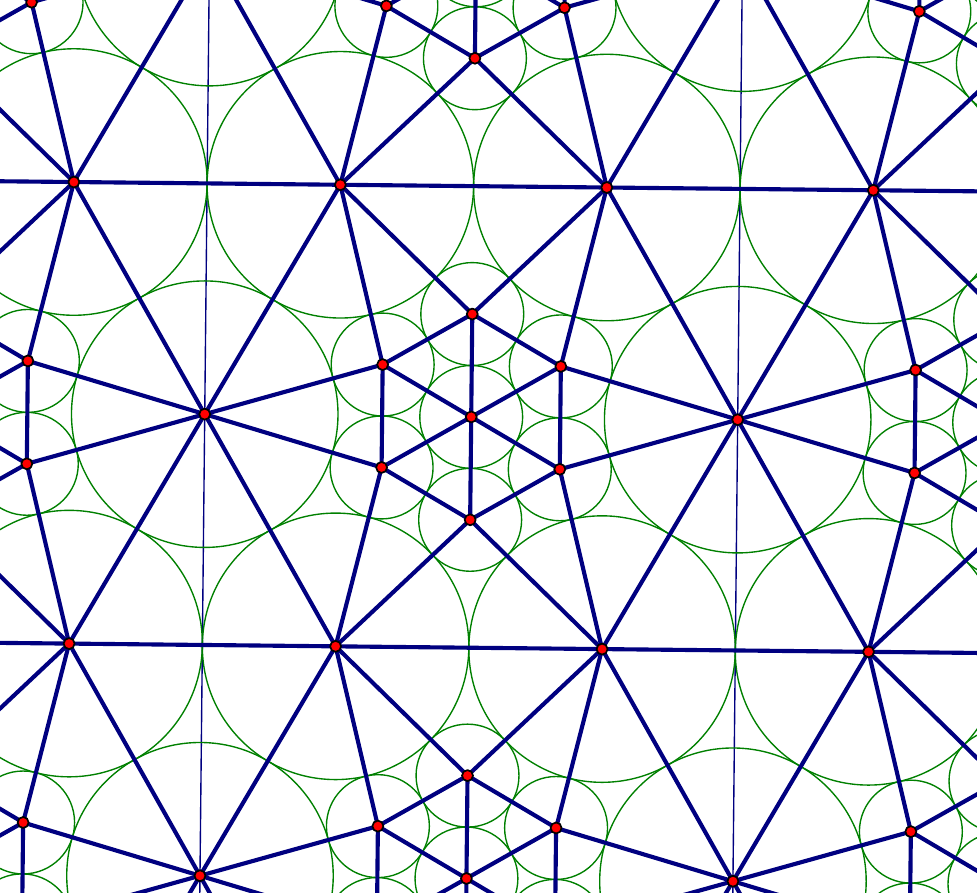}%
        \captionsetup{labelsep=colon,margin=2cm}   
    \caption{This is a packing of $3$ large disks and $7$ smaller ones for the given triangulation.}
    \label{fig:Kennedy-5}
     \end{figure}
     
The triangulation of Figure \ref{fig:Kennedy-5} has a $6$-fold rotational symmetry about the center of the central small circle, so the $3$ large circles have the same radius.  Similarly the $6$ small circles adjacent to the central small circle have the same radius.  Since the central circle is adjacent to the $6$ other small circles, it must have the same radius as the other small circles.  This is a compact/triangulated packing, one of $9$ possible classes described by Tom Kennedy in \cite{Kennedy-2-size}.

A similar analysis can be done in the tricusp case where some packings are conjectured to be the most dense by Uche, Stillinger and Torquato in \cite{Torquato-tricusp} using the $3$-fold symmetry.

There are three types of motions that increase the packing density. Each can be implemented with a Monte Carlo-type process, or a linear programming algorithm.
\begin{enumerate}
	\item(Danzer)  The lattice defining the torus metric is fixed while the configuration is perturbed so that the radii can be increased uniformly. \cite{Danzer}
	\item(Swinnerton-Dyer)  The lattice is deformed decreasing its determinant (and therefore the area of the torus) adjusting the configuration while fixing the radii. \cite{Swinnerton-Dyer}
	\item(Thurston)  The radii are adjusted fixing the configuration and the lattice so that the packing condition is preserved while increasing the sum of the squares of the radii.  This is essentially maximizing a positive definite quadratic function subject to linear constraints.  \cite{Thurston}
\end{enumerate}

The idea is that one can perform each of these motions, separately or together depending on what is desired.  Each process is named after a person who promoted that process in one form or another. 

\section{Conjecture} \label{section:conjectures}

 Kennedy in \cite{Kennedy-2-size} points out that there are triangulated packings of the plane (and effectively for a flat torus) that are not the most dense for given ratio of radii, which was $\sqrt{2} -1$ in the case being considered.  The idea is to take a square and equilateral tiling of the plane, use the vertices of that tiling for the centers of the larger disks, and the centers of the squares for the centers of the smaller disks.  So the final density of the packing is a weighted average of $\pi/\sqrt{12}=0.906899..$, the density of the triangular close packing, and $\pi(2-\sqrt{2})/2=0.920151..$, the density of the packing in Figure \ref{fig:best-binary}.  Figure \ref{fig:mixed-binary} shows such a periodic triangulated packing with density less than the maximal density.  The point is that even if we have a triangulated packing by disks of various sizes, that does not insure that it necessarily represents the maximum density for those sizes.
 
 We say that a packing of disks is \emph{saturated}  if there is no place to insert one of the disks in another part of the packing.

\begin{figure}[h]
    \centering
        \includegraphics[width=0.6\textwidth]{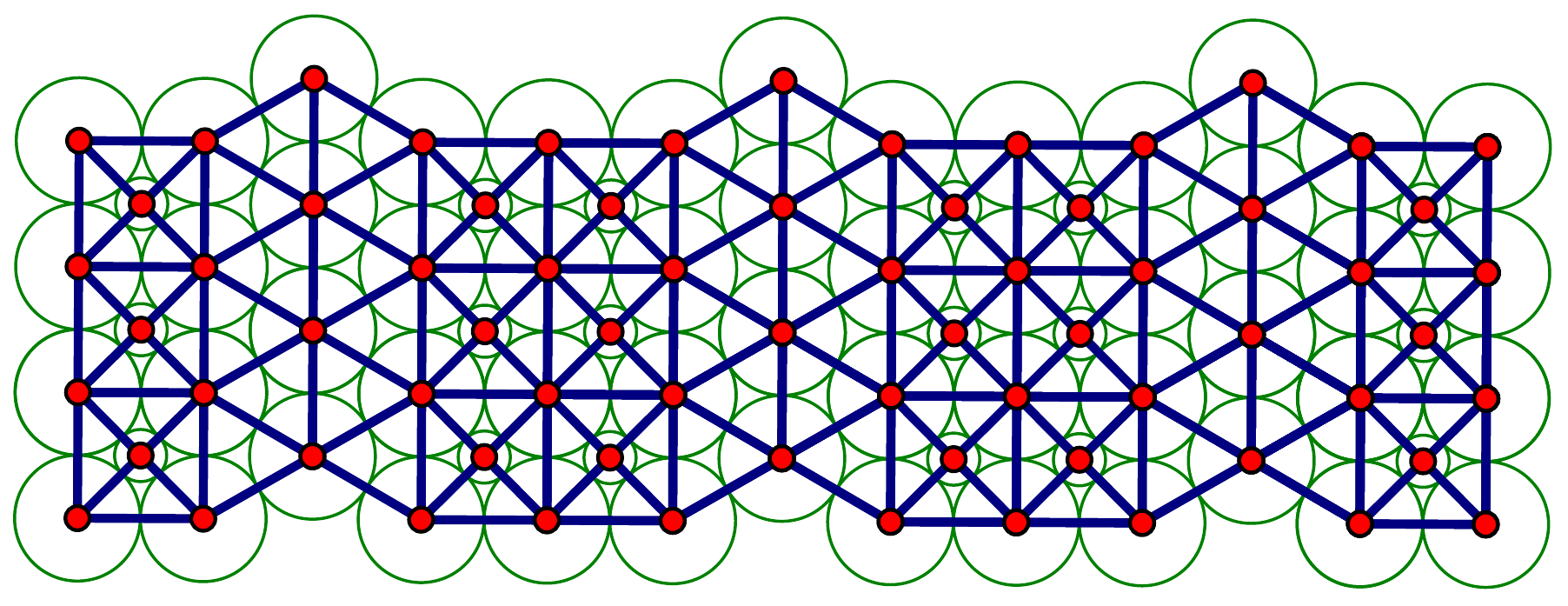}%
        \captionsetup{labelsep=colon,margin=2cm}   
    \caption{This is a periodic binary packing of disks with radius ratio $\sqrt{2}-1$ with density less than the maximum possible $\pi(2-\sqrt{2})/2=0.920151..$.}
    \label{fig:mixed-binary}
     \end{figure}
     
Taking a big leap, nevertheless, we conjecture the following:

\begin{conjecture} \label{conjecture:density} Suppose that ${\bf P}$ is a saturated packing, with a triangulated graph, of a finite number of packing disks in a torus with $n_1, n_2, \dots, n_k$ disks of radius $r_1 >r_2, \dots > r_k$ respectively with density $\rho_0$.  Then for all integers $m \ge 1$, and a packing of a torus with $mn_1, mn_2, \dots,m n_k$ disks of radius $r_1 >r_2, \dots > r_k$ respectively, the density is $\rho \le \rho_0$.
\end{conjecture}

Originally the condition that the packing was saturated was omitted, and Fedja Nazarov found the following counterexample, shown in Figure \ref{fig:Nazarov}.  

\begin{figure}[h]
    \centering
        \includegraphics[width=0.8\textwidth]{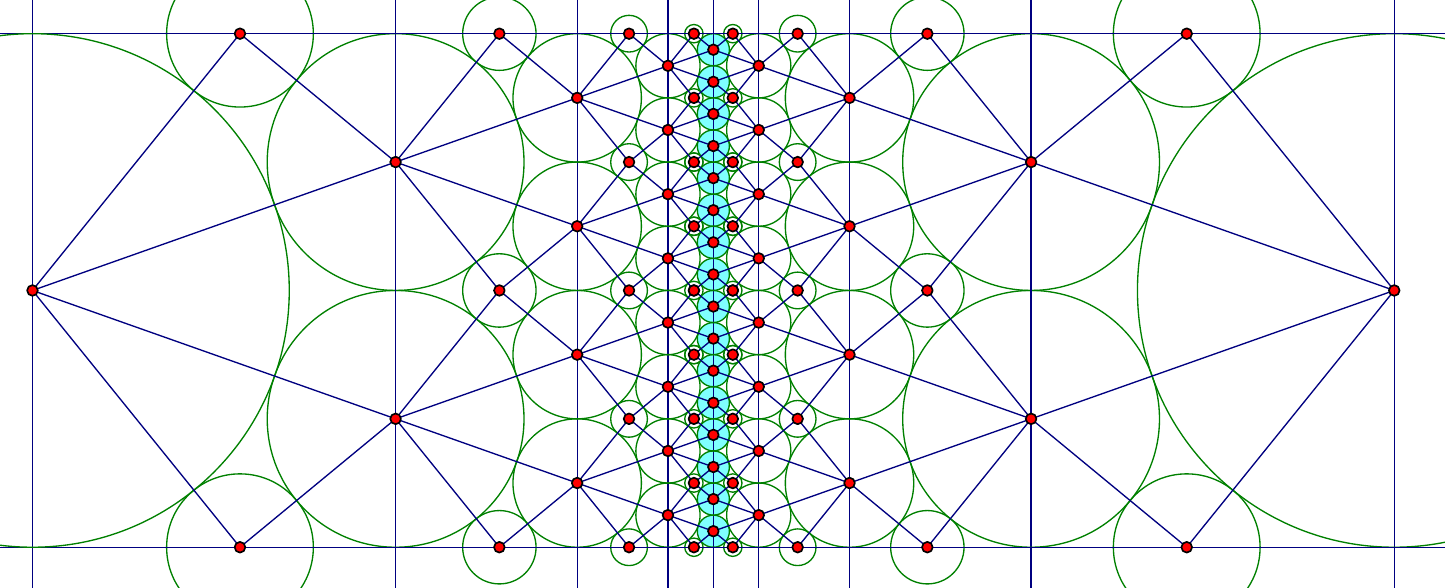}%
        \captionsetup{labelsep=colon,margin=2cm}   
    \caption{This is a fundamental region of a triangulated packing of a torus, where the line of small colored packing disks can be removed and reinserted in the triangular regions to the right and left.  Then the packing becomes not collectively jammed.  Indeed it is not even locally maximally dense, since the packing disks have room to grow into the line of removed disks.}
    \label{fig:Nazarov}
     \end{figure}
     
In \cite{Kennedy-2-size}, Kennedy shows a list of nine classes of all the triangulated packings of disks in the plane with just two disk sizes.  Seven of those nine packing have been shown, by Heppes and Kennedy \cite{Heppes-2-size-densest, Kennedy-densest} to be the most dense using just those two sizes, which is a bit stronger than the statement of Conjecture \ref{conjecture:density}.  This is support for Conjecture \ref{conjecture:density}.

An interesting special case is a packing of a torus with  $n_1$ disks of radius $1$ and $n_2$ disks of radius $\sqrt{2} -1$, $n_1 \ge n_2$, then Conjecture \ref{conjecture:density} implies that the maximum density is 
\[
\rho = \pi\frac{n_1 + n_2(\sqrt{2}-1)^2}{2\sqrt{3}(n_1-n_2) +4n_2}.
\]
Notice in the case when $k=2$, and $r_2/r_1 =\sqrt{2}-1$, and $n_1=n_2$ the statement of Conjecture \ref{conjecture:density} is weaker that Heppes's Theorem  \cite{Heppes-2-size-densest}, since it assumes $n_1=n_2$.  On the other hand, as far as we know, for other proportions of sizes of disks, Conjecture \ref{conjecture:density} is not known.  In particular, continuing with the  $r_2/r_1 =\sqrt{2}-1$, $n_2 > n_1$  case, it seems that a triangulated packing of a torus does not exist, and perhaps the most dense packings segregate into the triangular lattice and square lattice pieces.   

In another direction, it would be interesting to see how the nature of the triangulation influences the density of the corresponding triangulated packing.  Given a triangulation of the plane, an \emph{ elementary stellar subdivision} is where a triangle in the triangulation is removed, and it is replaced by the cone over its boundary, or an edge is removed and replaced by the cone over the resulting quadrilateral.  For a stellar subdivision of a triangle, it is clear that the density of the corresponding must increase, since one simply places an additional disk in the given triangular region.  In many cases, for the stellar subdivision of an edge the density increases.  However, if one starts with the Heppes packing graph of Figure \ref{fig:best-binary} and does a stellar subdivision as indicated in Figure \ref{fig:stellar-deflation}, the density decreases.  Indeed, after another stellar subdivision one gets back to a two-fold covering of the the original Figure \ref{fig:best-binary} with the same density.

\begin{figure}[h]
    \centering
        \includegraphics[width=0.3\textwidth]{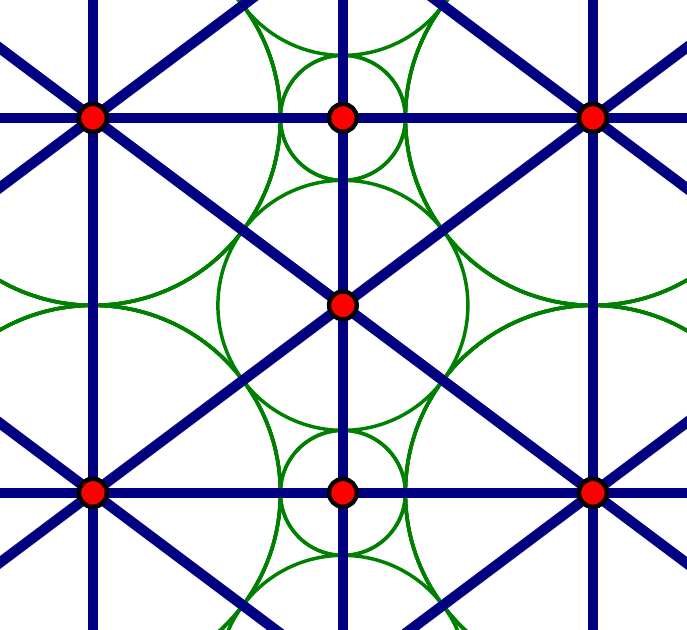}%
        \captionsetup{labelsep=colon,margin=2cm}   
    \caption{This a stellar subdivision of the graph of Figure \ref{fig:best-binary}, and has density $7\pi/24=0.91629..$, which is somewhat less than $\pi(2-\sqrt{2})/2=0.92015..$, the density of Figure \ref{fig:best-binary}. The radii are in ratio $1:2:3$.}\label{fig:stellar-deflation}
   \end{figure}

\section{Remarks and related work} \label{section:remarks}

There appear to be roughly four groups who work on packing problems, each from their particular point of view.  

One group deals with simply finding dense packings of circles in the plane and proving certain packings are the most dense when possible.  This group is epitomized by the work of L\'aszl\'o Fejes T\'oth.  His book \cite{Fejes-Toth-book} is an early attempt to show what was known and conjectured along with many other related problems and conjectures.  Indeed, there are places in \cite{Fejes-Toth-book} where it seems that Conjecture \ref{conjecture:density} is essentially in the background, at least for some particular packings.  See Melissen's Phd. Thesis \cite{Melissen} for quite a few conjectures for the most dense packings in various containers.  Note that the most dense packings conjectured  in \cite{Connelly-Vivian} agree with those in \cite{Melissen}.  For provably most dense packings for  fixed square and triangular tori, for small numbers of disks there are the results in \cite{Connelly-Dickinson, Musin, Connelly-Vivian}.

A second group uses linear programing techniques to find upper bounds for sphere packings, particularly in higher dimensions, where things are generally much harder.  A good outline of this point of view in the survey by de Laat, Filho and Vallentin \cite{packing-survey}, where their techniques work in dimension two as well.  However, their bounds are often not sharp.  For example, in dimension three for binary sphere packings where the ratio of the radii is $\sqrt{2}-1$, the most dense packing is conjectured to be $0.793$ which is achieved when the large spheres are centered at the face centered cubic lattice.  These centers form the vertices of a tiling of space by regular octahedra and regular tetrahedra.  The centers of small spheres are placed at the centers of the octahedra, so that the graph of the packing forms the one-skeleton of a triangulation of space.  This is the structure of NaCl ordinary table salt.  The techniques of \cite{packing-survey} provide an upper bound of $0.813$, which is reasonably close to the salty lower bound of $0.793$.  The salt packing is a three-dimensional extension of the configuration of Figure \ref{fig:best-binary}.

A third group has to do with the Koebe-Andreev-Thurston algorithms that creates packings from the graph of the packings.  However, these techniques do not initially specify the radii or density of the resulting packing.  A good overview is in the book by Ken Stephenson \cite{Stephenson}, where many examples are shown as well as connections to conformal mappings, etc.  It is interesting to note that one of proofs of the Koebe-Andreev-Thurston packings comes from a minimization argument by Colin de Verdi\`ere in \cite{Colin-de-Verdiere}, similar to the process that described in Section \ref{section:varying}.

A fourth group is motivated from the physics of granular material or colloidal clusters at in the work of Torquato and others in \cite{Torquato-isostatic, Donev-Torquato-Connelly}.

There is a lot of room for generalization and possibly improvement of the results here.

\begin{enumerate}
\item  Is the example Figure \ref{fig:2-disks-square} the only case for $n \ge 2$ disks on a fixed rectangular torus, where the isostatic condition does not hold?  It should be kept in mind that the results in \cite{Torquato-rattlers}  shows that non-isostatic (strictly) jammed packings seem to be quite frequent for larger numbers when there is one size of radius.

\item It seems very reasonable that an analysis similar to the argument here shows that a corresponding isostatic conjecture holds, where the parameters of a compact hyperbolic $2$-dimensional surface are generic, as well as the ratio of the radii.

\item The proof that an infinitesimal flex implies a finite motion for packings of disks on the $2$-dimensional sphere is not known, yet it seems reasonable that it is true, and if so, there should be a corresponding isostatic condition for generic radii.

\item There are many circumstances where there is a jammed packing in a bounded container with an appropriate condition on the boundary of the container, and the infinitesimal flex implies finite motion.  It is reasonable that if the shape of the container is generic including the ratio of the radii, that the packing is isostatic.   

\item Is there a way to prove the isostatic conjecture for packings in higher dimensions.  Presumably this would involve a different argument, since the analytic packing theory would not be available.  In the $2$-dimensional case, the process described here, with radii, configuration centers, and lattice moving, eventually converges to the case when the faces are all triangles.  In higher dimensions, it might be the case that the triangles are replaced by rigid polytopes as in the discussion in \cite{Bezdeks-Connelly}.
\end{enumerate}

\bibliographystyle{plain}

\end{document}